\documentclass[11pt, twoside]{article} 
\usepackage[left=3cm,right=3cm,top=3cm, bottom=4cm,asymmetric]{geometry}

\RequirePackage[OT1]{fontenc}
\RequirePackage{amsthm,amsmath,amsfonts,amssymb}
\RequirePackage[numbers]{natbib}
\RequirePackage[colorlinks,citecolor=blue,urlcolor=blue]{hyperref}
\RequirePackage{hypernat}
\usepackage{epsfig, subfigure}
\usepackage{afterpage}
\usepackage{color}

\usepackage{fancyhdr}
\usepackage[all]{xy}
\theoremstyle{plain} \newtheorem{lem}{Lemma}[section]
\theoremstyle{plain} \newtheorem{prop}[lem]{Proposition}
\theoremstyle{plain} 
\theoremstyle{plain} 
\theoremstyle{plain} 
\theoremstyle{plain} \newtheorem{thm}[lem]{Theorem}
\theoremstyle{plain} \newtheorem{cor}[lem]{Corollary}
\theoremstyle{plain} 
\theoremstyle{definition} \newtheorem{defn}[lem]{Definition}
\theoremstyle{definition}
\theoremstyle{definition} \newtheorem{rem}[lem]{Remark}
\theoremstyle{definition}\newtheorem{ex}[lem]{Example}
\theoremstyle{definition}\newtheorem{prob}[lem]{Problem}

\newcommand{\R}{\mathbb{R}}
\newcommand{\C}{\mathbb{C}}

\newcommand{\union}{\cup}

\def\independent{\!\perp\!\!\!\perp\!}
\def\given{\,|\,}
\newcommand{\RLCT}{{\rm RLCT}}

\newcommand{\corr}{\operatorname{corr}}

\newlength\savedwidth

\begin{document}
\thispagestyle{plain}
	\begin{center}
	{\bf{\LARGE{Hypersurfaces and their singularities \\
in partial correlation testing}}}

	\vspace*{.1in}
	\begin{tabular}{c}
Shaowei Lin$^{1}$, Caroline Uhler$^{2}$, Bernd Sturmfels$^{1}$ and Peter B\"uhlmann$^{3}$ \\ \quad \\
$^{1}$Department of Mathematics, UC Berkeley\\
$^{2}$IST Austria\\
$^{3}$Seminar for Statistics, ETH Z\"urich\\
	\end{tabular}
	\vspace*{.1in}
	\end{center}

\begin{abstract}
\noindent
An asymptotic theory is developed for computing volumes of
regions in the parameter space of a
directed Gaussian graphical model  that are
obtained by bounding partial correlations.
We study these volumes using the  method
 of real log canonical thresholds from algebraic geometry.
Our analysis involves the computation
of  the  singular loci of
correlation hypersurfaces. 
Statistical applications include the strong-faithfulness 
assumption for the PC-algorithm, and the quantification of 
confounder bias in causal inference.
A detailed analysis is presented for
trees, bow-ties, tripartite graphs, and complete graphs.
\let\thefootnote\relax
\end{abstract}

\section{Introduction}
\label{sec:intro}
Extensive theory has been established in recent years for causal
inference based on
 directed acyclic graph (DAG) models. A popular way for estimating a DAG model from observational data
  employs  partial correlation testing to infer the conditional independence relations in the model. In this paper, we
  apply algebraic geometry and singularity theory to analyze partial
  correlations in the Gaussian case.
   The objects of our study are algebraic hypersurfaces
 in the parameter space of a given graph that
 encode conditional independence~statements.

We begin with definitions for
graphical models in statistics.
A DAG is a pair $G = (V,E)$ consisting of a set $V$ of nodes
and a set $E$ of directed edges with no directed cycle.
We usually take $V=\{1,2,\ldots ,p\}$ and we associate random variables 
$X_1,X_2,\ldots,X_p$ with the nodes.  Directed edges are denoted by $(i,j)$ or
$i\rightarrow j$. The \emph{skeleton} of a DAG $G$ is the
underlying undirected graph obtained by removing the arrowheads. A node $i$ is  an \emph{ancestor} of $j$ if there is a directed path $i\rightarrow \cdots \rightarrow j$, and a configuration $i\to k$, $j\to k$ is  a \emph{collider} at $k$. Finally, we assume that the vertices are topologically ordered, that is $(i,j) \in E$ implies~$i < j$.
 

Every DAG $G$ specifies a {\em Gaussian graphical model} as follows.
The adjacency matrix $A_G$ is
the strictly upper triangular matrix whose entry in row $i$ and column $j$
is a parameter $a_{ij}$ if $(i,j) \in E$ and it is zero if $(i,j)
\not\in E$. The Gaussian graphical model is defined by the structural equation model $X=A_G^T X+\epsilon$, where $X=(X_1,\ldots , X_p)^T$. We assume that $\epsilon\sim\mathcal{N}(0,I)$, where $I$ is the $p \times p$-identity matrix. Then the \emph{concentration matrix} of this model equals
$$ K \, = \, (A_G - I)(A_G-I)^T.$$ 
Since ${\rm det}(K)=1$, the covariance matrix $\Sigma = K^{-1}$ 
is equal to the adjoint of $K$. 
The entries of the symmetric matrices $K$ and $\Sigma$ are polynomials in the parameters
$a_{ij}$. Our parameter space for this DAG model
will always be a  full-dimensional subset $\Omega$ of $\R^{|E|}$.

For any subset $S \subset V$ and distinct elements $i,j$ in $V \backslash S$,
we  represent the conditional independence statement 
$i \independent j \given S$ by an {\em almost-principal minor}
of either $K$ or $\Sigma$. By this we mean a square submatrix whose
sets of row and column indices differ in exactly one element.
To be precise, $\,i \independent j \given S\,$ holds for the multivariate normal distribution
with concentration matrix $K$  if and only if the submatrix 
$K_{iR,  jR} $ is singular, where $R = V \backslash (S \cup \{i,j\})$ and $iR=\{i\}\cup R$.
The determinant~$\,{\rm det}(K_{iR, jR})\,$ is a polynomial in $(a_{ij})_{(i,j)\in E}$. 
We are interested in the hypersurface in $\R^{|E|}$ defined by the vanishing of this polynomial.
Indeed, the {\em partial correlation} $\corr(i,j|S)$ is up to sign equal to the algebraic expression
\begin{equation}
\label{eq:parcor}
\frac{{\rm det}(K_{iR, jR})}{\sqrt{
{\rm det}(K_{iR,iR}) \cdot {\rm det}(K_{jR,jR})}}. 
\end{equation}
Since  the principal minors under the square root sign are strictly positive,
$\corr(i,j|S) = 0$ if and only if $\,{\rm det}(K_{iR, jR})
= 0$. If this holds for all $a \in \R^{|E|}$  then 
 $\,i \independent j \given S\,$ for $G$ and we say
that $i$ is \emph{d-separated} from $j$ \emph{given} $S$. This translates into
 a combinatorial condition on the graph $G$ as follows \cite[\S 2.3.4]{SGS}.
   An undirected path $P$ from $i$ to $j$
             \emph{d-connects} $i$ and $j$
             \emph{given} $S$~if
 \begin{enumerate}
 \item[(a)] every non-collider on $P$ is not in $S$,
\item[(b)] every collider on $P$ is in $S$ or an ancestor of a node in $S$.
\end{enumerate}
If $G$ has no path that d-connects $i$ and $j$ given $S$, then $i$ and $j$ are
{\em d-separated given} $S$,
and ${\rm det}(K_{iR,jR}) \equiv 0$ as a function of $a$.
  The \emph{weight} of a path $P$ is the product of all edge weights 
$a_{rs}$ along this path. It was shown in \cite[Equation (11)]{URBY} that the numerator
${\rm det}(K_{iR, jR})$ in (\ref{eq:parcor}) is a linear combination,
as in (\ref{eq:fsosf}),
 of the weights of all paths that d-connect $i$ to $j$  given~$S$.

Our primary  objects of study are the following subsets
of the parameter space:
\begin{equation}
\label{eq:tube_ijS}
{\rm Tube}_{i,j|S}(\lambda) \,\,= \,\, \{ \,\omega \in \Omega : |\corr(i,j|S)| \leq \lambda \,\}.
\end{equation}
Here $\corr(i,j|S)$ is a function of the parameter $\omega$ (denoted $(a_{ij})_{(i,j)\in E)}$ above) in the space $\Omega \subset \R^{|E|}$, $\lambda$ is a parameter in $[0,1]$, and
 $(i, j,S)$ is a triple where $i$ and $j$ are d-connected given $S$. These ``tubes'' can be seen as hypersurfaces which have been fattened up by a factor which depends on $\lambda$ and the position on the hypersurface
(see Figure \ref{fig:tubes}).
The volume of  
$\,{\rm Tube}_{i,j|S}(\lambda)\,$ with respect to a given measure, or {\em prior},
 $\,\varphi(\omega)\,d\omega\,$ on $\,\Omega \subset \R^{|E|}\,$ is represented by the integral
\begin{equation}
\label{eq:volume1}
{\rm V}_{i,j|S}(\lambda) \,\,= \,\,\int_{{\rm Tube}_{i,j|S}(\lambda)} \varphi(\omega)\,d\omega.
\end{equation}
In this paper we study the asymptotics of this integral
when the parameter $\lambda$ is close to~$0$.

Two applications in statistics are our motivation.
The first  concerns the strong-faithfulness assumption for algorithms that learn
 Markov equivalence classes of DAG models by inferring conditional independence relations.
The PC-algorithm \cite{SGS} is a prominent instance.
Our set-up is exactly as in  \cite{URBY}.
The Gaussian distribution with concentration
matrix $K$ is $\lambda$-\emph{strong-faithful} to a DAG $G$ if, for  any 
 $S \subset V$ and $i,j \notin S$, we have
$|\corr(i,j|S)\,| \leq \lambda$ if and only if
$i$ is d-separated from $j$ given $S$.  We write $V_G(\lambda)$ for the volume
of the region in $\Omega$ representing distributions that are not
$\lambda$-strong-faithful. In other words, $V_G(\lambda)$ is the volume of the union of all tubes 
in $\Omega$ that correspond to non d-separated triples $(i,j,S)$.

 Zhang and Spirtes~\cite{ZhangSpirtes03} proved uniform consistency of the PC-algorithm under the
strong-faithfulness assumption with $\lambda \asymp 1/\sqrt{n}$, 
provided the number of nodes $p$ is fixed and sample size
$n \to \infty$. In a high-dimensional, sparse setting, Kalisch and
B\"uhlmann~\cite{kabu07} require strong-faithfulness with
$\lambda \asymp \sqrt{\deg(G)\log(p)/n}$, where $\deg(G)$ denotes the maximal degree (i.e., sum of indegree and outdegree) of nodes in~$G$. 

In order to understand the properties of the PC algorithm for a large sample size $n$, it is essential to determine the asymptotic behavior of the unfaithfulness volume
$V_G(\lambda)$   when $\lambda$ tends~to~$0$. Given a prior $\varphi$ over the
parameter space, $V_G(\lambda)$ is the prior probability that the true
parameter values violate $\lambda$-strong faithfulness. Thus $1-V_G(\lambda)$ for $\lambda \asymp 1/\sqrt{n}$ describes the prior probability that the PC algorithm is able to recover the true graph. 
 We shall see in Example \ref{ex:specialcases}
   that $V_G(\lambda)$ depends  on the choice of the parameter space
   $\Omega$ and the prior $\varphi$.


 We shall address the issue of computing $V_G(\lambda)$ as $\lambda\to 0$.
 This will be done    using  the concept of
   {\em real log canonical thresholds}   \cite{AGV,LinThesis,Watanabe09}.
Our Section~\ref{sec:real_log_threshold} establishes the
existence of  positive constants $\ell,m,C$ (which depend on $G$ and $\varphi$)
  such that, asymptotically for $\lambda \rightarrow 0$,
\begin{equation}
\label{eq:volume2}
\begin{matrix}
{\rm V}_G(\lambda) && \approx &&
 C \cdot \lambda^{\ell} \cdot (- \ln \lambda)^{m-1} ,\\
{\rm V}_{i,j|S}(\lambda) && \approx &&
 C'  \cdot \lambda^{\ell'} \cdot (- \ln \lambda)^{m'-1}. 
 \end{matrix}
 \end{equation}
(See (\ref{eq:curly_equal}) for an exact definition of  $\approx$.) This refines the results in \cite{URBY} on
the growth of $V_G(\lambda)$ via the geometry of
 the {\em correlation hypersurfaces} 
  $\,\{ {\rm det}(K_{iR, jR}) = 0\}$.
While \cite{URBY} focused on developing bounds on $V_G(\lambda)$ for the low-dimensional as well as the high-dimensional case and showed the importance of the number and degrees of these hypersurfaces,
we here analyze the exact asymptotic behavior of $V_G(\lambda)$ for $\lambda\to 0$ and $G$ fixed and demonstrate the importance of the singularities of these hypersurfaces.
 Singularities get fattened up much more than smooth parts of the hypersurface,
 and this increases the volumes (\ref{eq:volume2})
  substantially.
 
Our second application concerns {\em stratification bias in causal inference} (see e.g.~\cite{Greenland, GP}).
Here, the volume $V_{i,j|S}(\lambda)$ being large
is not a problematic feature, but is in fact desired. 
 Suppose we want to study the effect of an exposure $E$ on a disease outcome $D$. 
 If there is an additional variable $C$ such that $D\rightarrow C\leftarrow E$, then stratifying (i.e.~conditioning) on $C$ tends to change the association between $E$ and $D$. This can lead to biases in effect estimation. 
  This is known as \emph{collider-bias}.
  On the other hand, if   $D\leftarrow C\rightarrow E$ holds, then
   $C$ is a \emph{confounder} and stratifying on $C$ corresponds to bias removal. In certain
 larger graphs, such as Greenland's {\em bow-tie} example \cite{Greenland},
 stratifying on $C$ removes confounder-bias but at the same time introduces collider-bias. 
 In order to decide whether one should stratify on such a variable $C$, it is important to understand the partial correlations involved. In this application, the volume  $V_{i,j|S}(\lambda)$ can be viewed 
  as the cumulative distribution function of the prior distribution of the partial correlation $\corr(i,j|S)$ implied by the prior distribution on the parameter space,
   and we compare the two cumulative distribution functions
   $V_{E,D|C}(\lambda)$ and  $V_{E,D}(\lambda)$.

In this paper we examine $V_{i,j|S}(\lambda)$ from a geometric perspective, and 
we demonstrate how this volume can be calculated 
using tools from singular learning theory. 
To derive the asymptotics (\ref{eq:volume2}),
the main player is the correlation hypersurface, which is
the locus in $\Omega$ where $\corr(i,j|S) $ vanishes.
The first question is whether this hypersurface
is smooth, and, if not, one needs to analyze the nature of its singularities.
We study these questions for various  classes of interesting causal models,
using methods from computational algebraic geometry.

\smallskip

The remainder of this paper is organized as follows: In Section~\ref{sec:graphs} we introduce
the families of DAGs which we will be working with throughout. 
Example \ref{ex:tripartite6} illustrates the algebraic computations that are involved in our analysis.
We also discuss some simulation results, which indicate the importance of singularities when studying the volume 
$V_G(\lambda)$ of strong-unfaithful distributions. \mbox{Section \ref{sec:real_log_threshold}} presents the connection to singular learning theory \cite{LinThesis,Watanabe09}
and explains how this theory can be used to compute the volumes of the tubes ${\rm Tube}_{i,j|S}(\lambda)$.  Example \ref{ex_tubes}
illustrates our theoretical results for some very simple polynomials
in two variables.

In \mbox{Section \ref{sec:sing_locus}} we develop algebraic algorithms for analyzing the singularities
  of the correlation hypersurfaces. We show that, for the
   polynomials ${\rm det}(K_{iR,jR} )$ of interest,
  the real singular locus  is often much simpler than the complex singular locus.
  For instance, Theorem~\ref{thm:uptosix} states that  these hypersurfaces are
  always smooth for complete DAGs with up to six nodes.
   In \mbox{Section \ref{sec:trees}} we study the singularities and the volumes
                       (\ref{eq:volume1}) for trees without colliders.  
                        
                        Section \ref{sec:bias}  focuses on our second
  application, namely bias reduction in causal inference.
  Problems \ref{prob:tripart} and \ref{prob:bowtie} offer
  precise versions of conjectures by Greenland \cite{Greenland},
  in terms of   comparing different volumes
    $V_{i,j|S}(\lambda)$ for fixed $G$.
    We establish some instances of these conjectures.

In Section \ref{sec:resolution} we introduce more advanced methods,
based on the resolution of singularities \cite{Hauser, Hironaka},
for finding the exponents $\ell$ and $m$ in  (\ref{eq:volume2}).
Finally, in Section \ref{sec:integrals} we present some new results on computing the
constants $C$ and $C'$ in our asymptotics
(\ref{eq:volume2}) for tube volumes.

\section{Four classes of graphs}
\label{sec:graphs}

\begin{figure}[b!]
\centering
\subfigure[$\textrm{Star}_p$]{\includegraphics[scale=0.28]{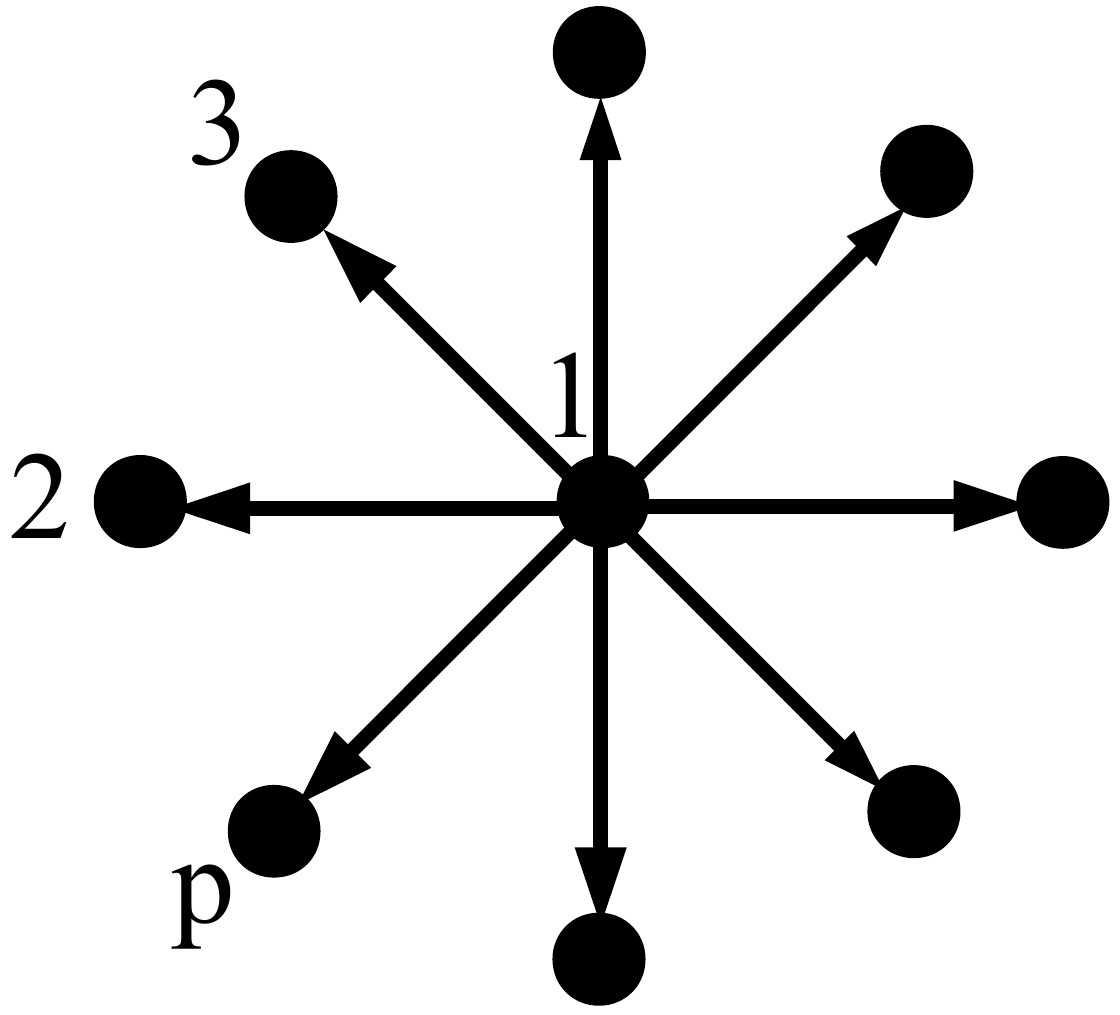}\label{fig:star}}\qquad\quad
\subfigure[$\textrm{Chain}_p$]{\includegraphics[scale=0.28]{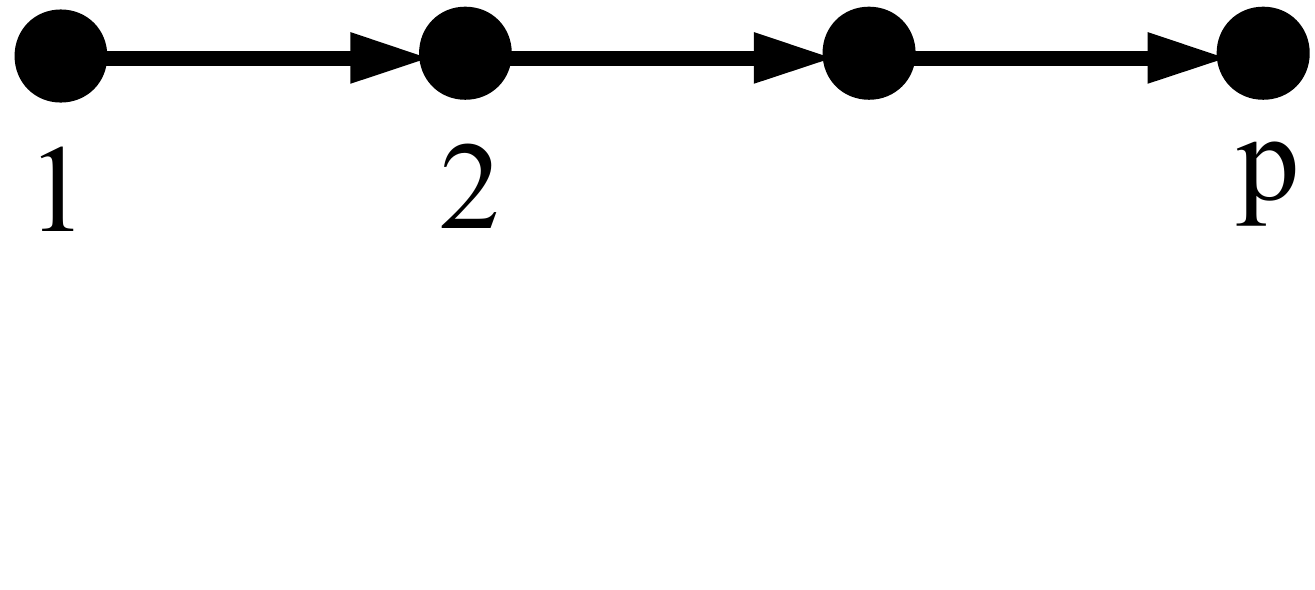}\label{fig:chain}}\qquad\quad
\subfigure[$\textrm{Tripart}_{p,p'}$]{\includegraphics[scale=0.28]{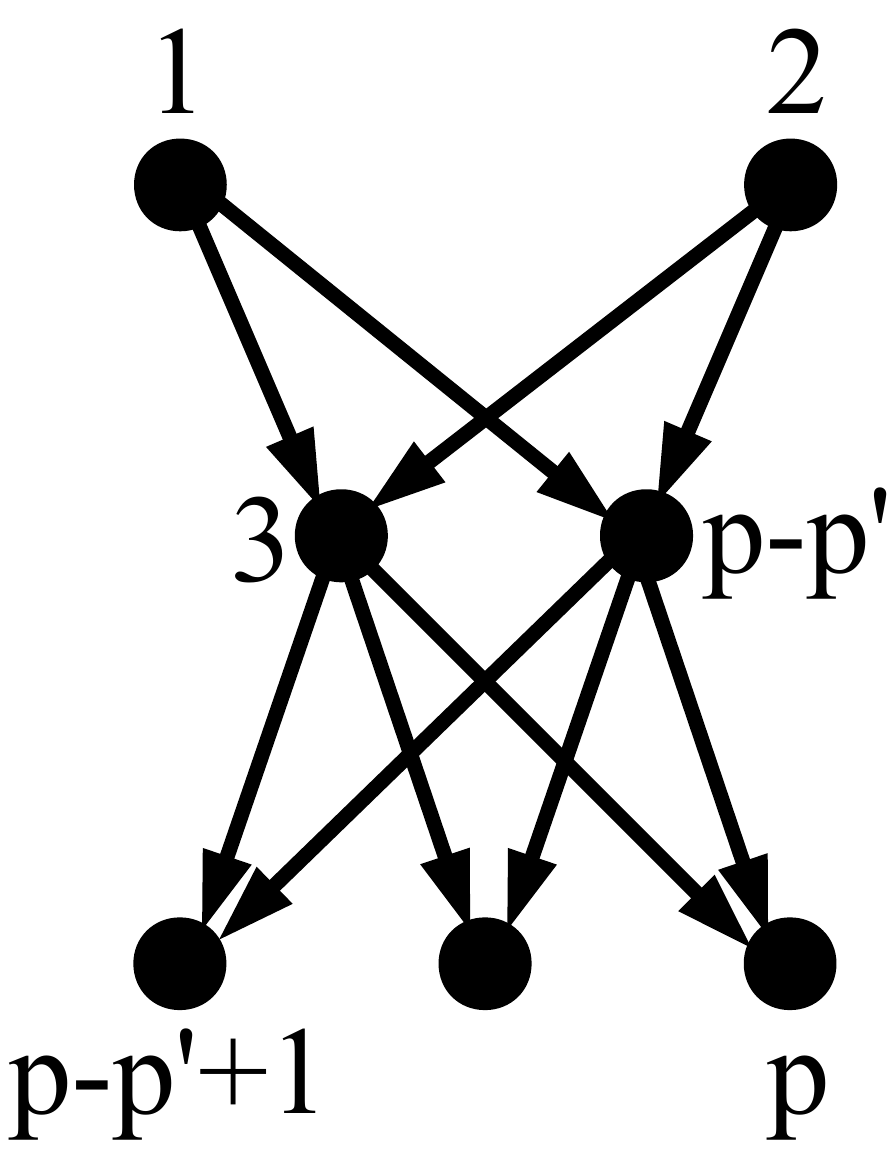}\label{fig:tripart}}\qquad\quad
\subfigure[$\textrm{Bow}_p$]{\includegraphics[scale=0.28]{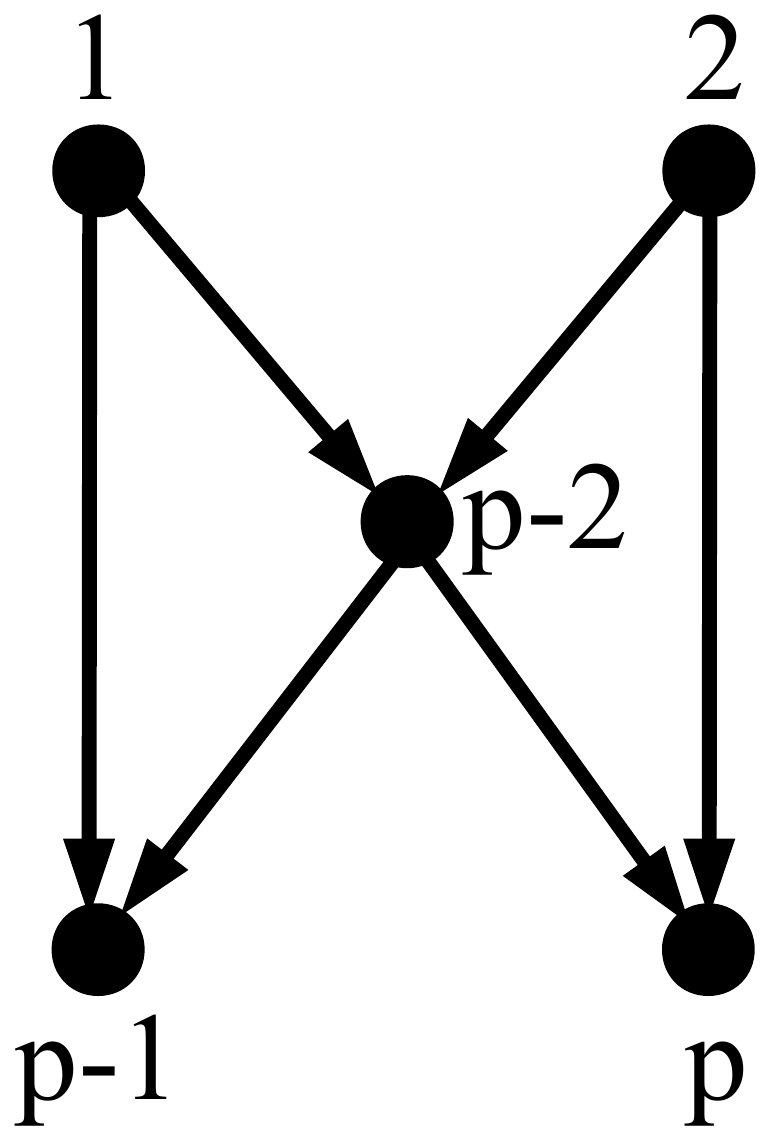}\label{fig:bowtie}}
\caption{Various classes of graphs.}
\label{fig:graphs}
\end{figure}

In this article we will be primarily working with four classes of DAGs:

\begin{enumerate}
\item[i)] \emph{Complete graphs}: We denote the complete DAG on $p$ nodes by $K_p$.
The corresponding matrix $A_{K_p}$ is strictly upper triangular and all
$\binom{p}{2}$ parameters $a_{ij}$ are present.
\item[ii)] \emph{Trees}: We call a DAG $G$ a tree graph if the skeleton of $G$ is a rooted tree and all edges point away from the root (i.e.~$G$ has no colliders). We are particularly interested in the most extreme trees, 
namely star and chain-like graphs. We denote the star graph shown in Figure \ref{fig:star} by $\textrm{Star}_p$ and the chain-like graph shown in Figure \ref{fig:chain} by $\textrm{Chain}_p$.
\item[iii)] \emph{Complete tripartite graphs}:  Let $A,B\subset V$ with $A\cap B=\emptyset$. Then we denote by $A\Rightarrow B$ the complete bipartite graph where $(a,b)\in E$ for all $a\in A$ and $b\in B$. A complete tripartite graph is denoted by $\textrm{Tripart}_{p,p'}$ with $1\leq p'\leq p-3$. It corresponds to the DAG $\{1,2\}\Rightarrow\{3,\dots ,p-p'\}\Rightarrow\{p-p'+1,\dots , p\}$ and is shown in Figure \ref{fig:tripart}.
\item[iv)] \emph{Bow-ties}: We define a bow-tie to be a complete tripartite graph $\textrm{Tripart}_{p,2}$ with two additional edges, namely $(1,p-1)$ and $(2,p)$. A bow-tie is denoted by $\textrm{Bow}_p$ and is shown in Figure \ref{fig:bowtie}. Bow-ties with $p=5$ feature prominently in Greenland's \mbox{study \cite{Greenland}.}
\end{enumerate}

The following example serves as a preview to the topics covered in this paper.

\begin{ex}
\label{ex:tripartite6}
We illustrate our objects of study for the 
tripartite graph $G = {\rm Tripart}_{6,2}$.
Since the error variances are assumed to be fixed at $1$,
 this DAG model has eight free parameters, namely the unknowns in the matrix
$$  A_G  \,\,= \,\,
\begin{pmatrix}
     \,0 & 0 & \! a_{13} \! & \! a_{14} \! & \! 0 \! & \! 0 \\
     \,0  &  0  & \! a_{23} \! & \! a_{24} \! & \! 0 \! & \! 0 \\
     \,0 &  0 & \! 0 \! & \! 0 \! &  \! a_{35} \! & \! \! a_{36} \\
     \,0  &  0 & \! 0 \! & \! 0 \! &  \! a_{45} \! & \! \! a_{46} \\
     \,0  &  0  & \! 0 \! & \! 0 \! & \! 0 \! & \! 0 \\
     \,0  &  0  & \! 0 \! & \! 0 \! & \! 0 \! & \! 0
     \end{pmatrix}.
 $$
 The covariance matrix $\Sigma$ equals the inverse (or the adjoint)  of
 the concentration matrix
$$
K \,\,= \,\,
\begin{pmatrix}
a_{13}^2+a_{14}^2+1 \! & \!   a_{13} a_{23}{+}a_{14} a_{24} \! & \!  -a_{13} \! & \!  -a_{14} \! & \!  0 \! & \!  0 \\
     a_{13} a_{23}{+}a_{14} a_{24} \! & \!  a_{23}^2+a_{24}^2+1 \! & \!  -a_{23} \! & \!  -a_{24} \! & \!  0 \! & \!  0 \\
     -a_{13} \! & \! -a_{23} \! & \!     a_{35}^2+a_{36}^2+1 \! & \!  a_{35} a_{45}{+}a_{36} a_{46} \! & \!\! -a_{35} \!\! & \! -a_{36} \\
     -a_{14} \! & \! -a_{24} \! & \!\!
      a_{35} a_{45}{+}a_{36} a_{46} \! & \!     a_{45}^2+a_{46}^2+1 \! & \!\!  -a_{45}\! \! & \! -a_{46} \\
     0 \! & \! 0 \! & \! -a_{35}  \! & \!-a_{45} \! & \! 1 \! & \! 0 \\
      0 \! & \! 0 \! & \! -a_{36} \! & \! -a_{46} \! & \! 0 \! & \!     1 
\end{pmatrix}.
$$
One conditional independence statement of interest is 
$\,1 \independent 2 \given \{5,6\}$.  
Its correlation hypersurface in 
$\R^8$ is defined by the almost-principal minor
in $\Sigma$ with rows $156$ and columns $256$, or 
the almost-principal minor in $K$
with rows $134$ and columns $234$. That determinant equals
\begin{equation}
\label{eq:fsosf}
\begin{matrix}
f &=& (1{+}a_{46}^2)a_{13} a_{23} a_{35}^2 + (1{+}a_{45}^2)a_{13} a_{23} a_{36}^2 +
 (1{+}a_{35}^2)a_{14} a_{24} a_{46}^2 + (1{+}a_{36}^2)a_{14} a_{24} a_{45}^2 \\
&& +a_{13} a_{24} a_{35} a_{45} +a_{13} a_{24} a_{36} a_{46} +a_{14} a_{23} a_{35}  a_{45}+a_{14} a_{23} a_{36}      a_{46}\\
&& - 2 a_{13} a_{23} a_{35} a_{36} a_{45} a_{46} - 2 a_{14}
  a_{24} a_{35} a_{36} a_{45} a_{46}.
\end{matrix}
\end{equation}
This is a weighted sum of all paths which d-connect nodes $1$ and $2$ given $\{5,6\}$. 
The first term in the formula (\ref{eq:fsosf}) for
$f = {\rm det}(K_{134,234})$
corresponds to the path $\,1\rightarrow 3\rightarrow 5\leftarrow 3 \leftarrow 2\,$
in $ G = {\rm Tripart}_{6,2}$,
 and the last term corresponds to the path $\,1\rightarrow 4 \rightarrow 5 \leftarrow 3 \rightarrow 6 \leftarrow 4 \leftarrow 2$.

Let $\varphi$ denote a prior on the parameter space. For this example 
we take $\varphi$ to be the Lebesgue probability measure on the cube $\Omega = [-1,+1]^8$.
The expression $V_{1,2|56} (\lambda)$ defined in (\ref{eq:volume1})
is the volume of
the region of parameters $a \in \Omega$ that satisfy 
$$|\corr(1,2\mid 5,6)|\;=\;\left|\frac{f(a)}{\sqrt{\det(K_{134,134})}\sqrt{\det(K_{234,234})}}\right| \;\leq\; \lambda.$$
As a function in $\lambda$, the  volume $V_{1,2|56} (\lambda)$ is a cumulative distribution function
on $[0, \infty)$. Our aim in this article is to determine the
 asymptotics of such a function for $\lambda \rightarrow 0$.

In Section~\ref{sec:real_log_threshold} we shall explain
the form of the asymptotics that is promised in  (\ref{eq:volume2}).
In order to find the exponents $\ell$ and $m$, the first
step is to run the algebraic algorithm in Section~\ref{sec:sing_locus}.
This answers the question whether the hypersurface in $\Omega$
defined by $f = 0$ has any singular points.
The set of such points, known as the {\em singular locus}, is
the zero set in $\Omega$ of the ideal
$$ J \quad = \quad 
\bigl\langle \,f\,,\,
 \frac{\partial{f}}{\partial{a_{13}}} ,
 \frac{\partial{f}}{\partial{a_{14}}} ,
 \frac{\partial{f}}{\partial{a_{23}}} ,
 \frac{\partial{f}}{\partial{a_{24}}} ,
  \frac{\partial{f}}{\partial{a_{35}}},
  \frac{\partial{f}}{\partial{a_{36}}},
  \frac{\partial{f}}{\partial{a_{45}}},
  \frac{\partial{f}}{\partial{a_{46}}} \,\bigr\rangle. $$
  The tools of Section~\ref{sec:sing_locus}
reveal that its {\em real radical} \cite{Marshall} 
is the intersection of three prime ideals:
\begin{eqnarray*}
\sqrt[\R]{J} &\!\! =\!\! &
\bigl\langle\, \hbox{entries of} \,
\begin{pmatrix}
a_{13} & a_{14} \\
a_{23} & a_{24} 
\end{pmatrix} \cdot
\begin{pmatrix}
a_{35} & a_{36} \\
a_{45} & a_{46}
\end{pmatrix}\, \bigr\rangle 
\\
&\!\! =\!\! &
\langle a_{13}, a_{14}, a_{23}, a_{24} \rangle \,\cap \,
\langle a_{35}, a_{36}, a_{45}, a_{46} \rangle \,\cap \,
\langle \hbox{$2{\times}2$-minors of} \,
\begin{pmatrix}
a_{13} &  a_{23} & a_{45} & a_{46} \\
 a_{14} & a_{24} & \! -a_{35} & \! -a_{36} \\
\end{pmatrix}
\rangle.
\end{eqnarray*}
Thus the hypersurface $\{f = 0\}$ is singular.
Its singular locus decomposes into three 
irreducible varieties, namely two linear spaces of dimension $4$
and one determinantal variety of dimension~$5$.
In Section \ref{sec:bias} we return to this example,  with 
focus on a statistical application of the cumulative distribution function
$V_{1,2|56} (\lambda)$. We will then show that $(\ell,m)$ equals $(1,1)$.
 \qed
\end{ex}

\begin{figure}[t!]
\centering
\subfigure[$p=6$]{\includegraphics[scale=0.4]{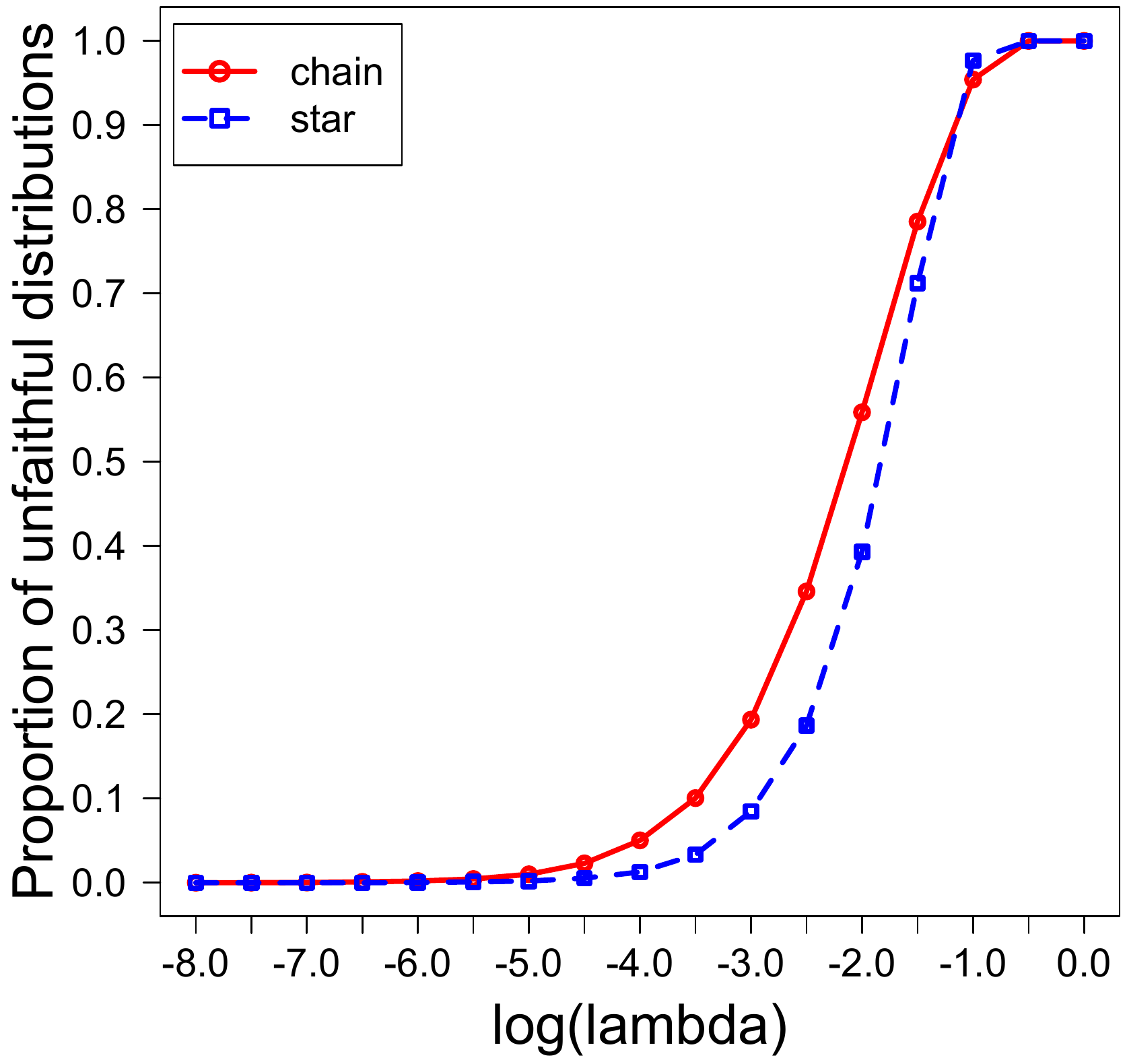}\label{Rplot_6}}\qquad\qquad
\subfigure[$p=10$]{\includegraphics[scale=0.4]{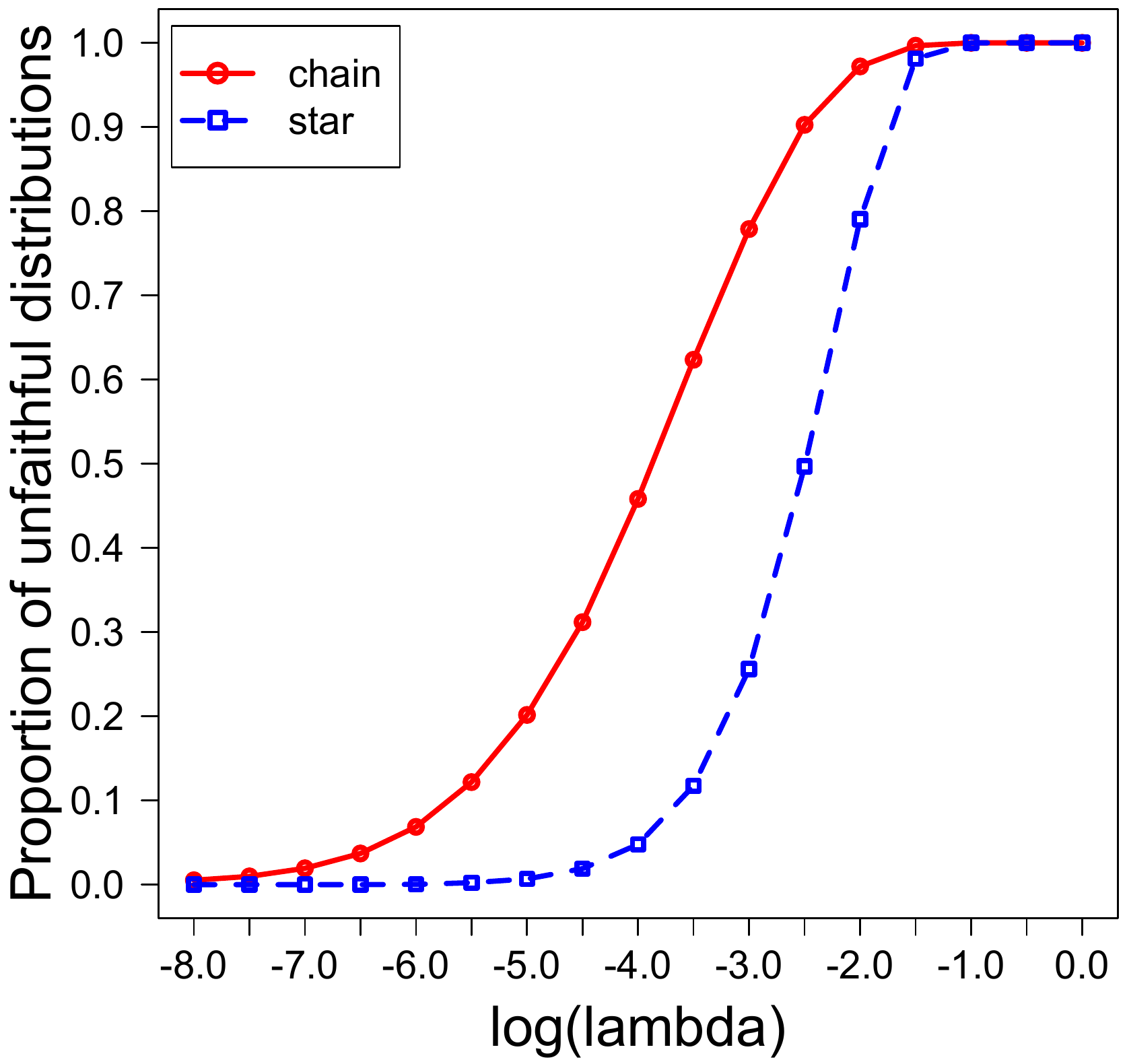}\label{Rplot_10}}
\caption{Proportion of $\lambda$-strong-unfaithful distributions for chains compared to stars.}
\label{fig:graphs_lambda}
\end{figure}

This paper extends the work 
of Uhler, Raskutti, B\"uhlmann and Yu in \cite{URBY}
on the geometry of the strong-faithfulness assumption
in the PC-algorithm.  Upper and lower bounds on the volume 
$V_G(\lambda)$ of the unfaithful region ${\rm Tube}_G(\lambda)$ for the low- as well as the high-dimensional setting
were  derived in \cite[\S 5]{URBY}. These bounds
involved only the number $|E|$ of parameters and the degrees
of the correlation hypersurfaces $\,\{{\rm det}(K_{iR,jR}) = 0\}$.
The new insight in the current paper is that singularities are essential  for the asymptotic behavior of $V_G(\lambda)$ for~$\lambda \rightarrow 0$.

What led us to this insight was taking a closer look at the
simulation results for trees. In \cite[\S 6.1.1]{URBY} trees were still
treated as one single class. We subsequently 
examined the difference between stars and chains, depicted in
Figure \ref{fig:star} and \ref{fig:chain}.
Our simulation results for ${\rm Star}_p$ and ${\rm Chain}_p$
are shown in Figure \ref{fig:graphs_lambda}.
We shall now explain the curves  in these diagrams.

The left diagram in Figure \ref{fig:graphs_lambda} is for $p = 6$ nodes
and the right diagram is for $p = 10$. Each curve is the graph
of the cumulative distribution function $V_G(\lambda)$ but with the
x-axis transformed to a logarithmic scale (with base 10). Thus we depict the
graph of the function 
\begin{equation}
\label{funct_plots}
(-\infty,0]\, \rightarrow \,[0,1],\,\, x\, \mapsto \,V_G(10^x) .
\end{equation}
The red curve is for $G = {\rm Chain}_p$ and the
blue curve is for $G = {\rm Star}_p$.
These curves were~computed by simulation:
we sampled the parameter $a$ from the uniform distribution on
$[-1,1]^{p-1}$ and we recorded the proportion of trials that landed in
${\rm Tube}_G(\lambda)$ for various values of $\lambda$.
The diagrams  show clearly that
 $V_G(\lambda)$ is smaller for star graphs than for chain-like
graphs. 

A theoretical explanation for these experimental results will be given in
Section \ref{sec:trees}. Our asymptotic theory predicts the behavior of these curves
as $x  = {\rm log}(\lambda)$ tends to $-\infty$. The point is that the correlation
hypersurfaces for chain-like graphs have deeper singularities than
those for star graphs.  The equation of any such hypersurface for a tree
is the product of a monomial and a strictly positive polynomial.
This enables us to apply Proposition \ref{prop:RLCTMonomial}.
In Theorem \ref{thm:treeV} and Corollary \ref{cor__const_trees}
we shall determine the constants $\ell$, $m$ and $C$ of
(\ref{eq:volume2}) exactly when the graph $G$ is a tree. We shall also address the question
of how to obtain $\ell$, $m$ and $C$ from simulations.

Before we get to graphical models, however, we first need
to develop the mathematics needed to analyze $V_G(\lambda)$.
 This will be done, in a self-contained manner,
in the next section.

\section{Computing the volume of a tube}
\label{sec:real_log_threshold}

We now introduce the basics regarding the computation of integrals
like the one in (\ref{eq:volume1}), and we explain why asymptotic formulas like (\ref{eq:volume2}) can be expected.
While this section is foundational for what is to follow, no reference to any statistical application is made until
 Theorem \ref{lemma_unfaithful}.  It can be read from first principle and might be of independent 
interest to a wider audience.

Let $\Omega \subset \R^d$ be a compact, full-dimensional, semianalytic subset
and consider a probability measure $\,\varphi(\omega) d\omega\,$ on
$\Omega$ where $d\omega$ is the standard Lebesgue measure and
$\varphi:\Omega \rightarrow \R$ is a real-analytic function. Also, fix an analytic function $f:\Omega \rightarrow \R$
whose hypersurface $\{\,\omega : f(\omega)=0\}\,$
has non-empty intersection with the interior of $\Omega$.
We are interested in the volume $V(\lambda)$ with respect to
the measure $\varphi$
of the region
$${\rm Tube}(\lambda)\,\,=\,\,\bigl\{\omega \in \Omega :|f(\omega)| \leq \lambda \bigr\}. $$
Here $\lambda > 0$ is a parameter that is assumed to be small.
In later sections, we often take $\Omega$ to be the cube
$[-1,+1]^d$, with $\varphi$ its Lebesgue probability measure, and
$f$ is usually a polynomial.

The asymptotics of the volume function $V(\lambda)$ depends on the singularities of the hypersurface 
$\{f = 0\}$. This phenomenon is illustrated in Figure \ref{fig:tubes}.
Our measure for the complexity of the singularities of $f$ is a 
 pair $(\ell,m)$ of non-negative real numbers. That pair is the
 {\em  real log canonical threshold} of $f$. It 
  is related to the volume $V(\lambda)$ 
   for small values of $\lambda$ by the formula
 \begin{equation}
 \label{eq:firstorder}
V(\lambda) \,\, \approx \,\, C\,\lambda^{\ell} (-\ln \lambda)^{m-1}.
\end{equation} 
Here $C$ is a positive real constant whose 
study we shall defer until Section \ref{sec:integrals}.

\begin{figure}[t!]
\centering
\subfigure[$f(x,y)=x$]{\includegraphics[scale=0.245]{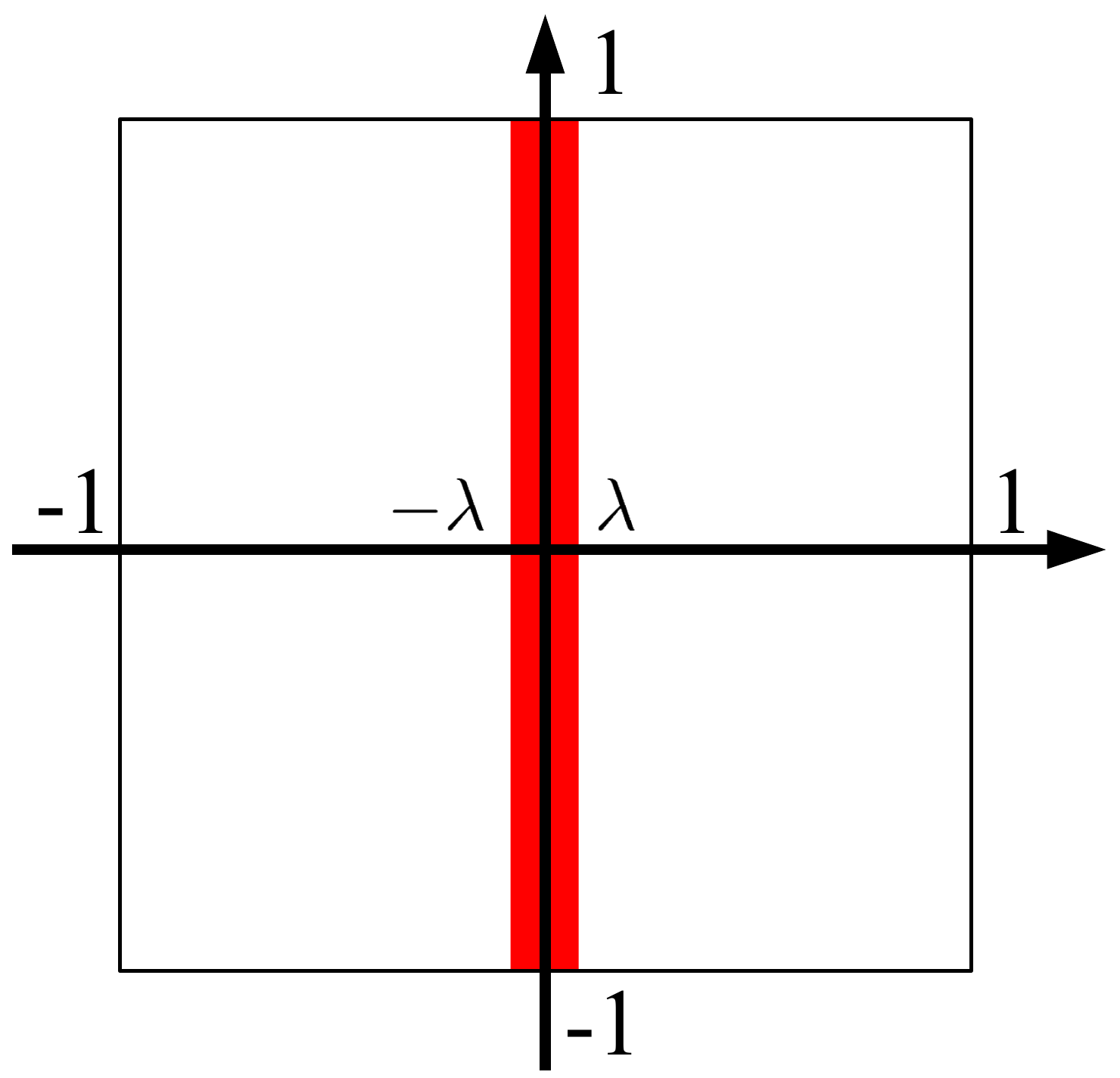}\label{fig:tubes_1}}\,
\subfigure[$f(x,y)=xy$]{\includegraphics[scale=0.245]{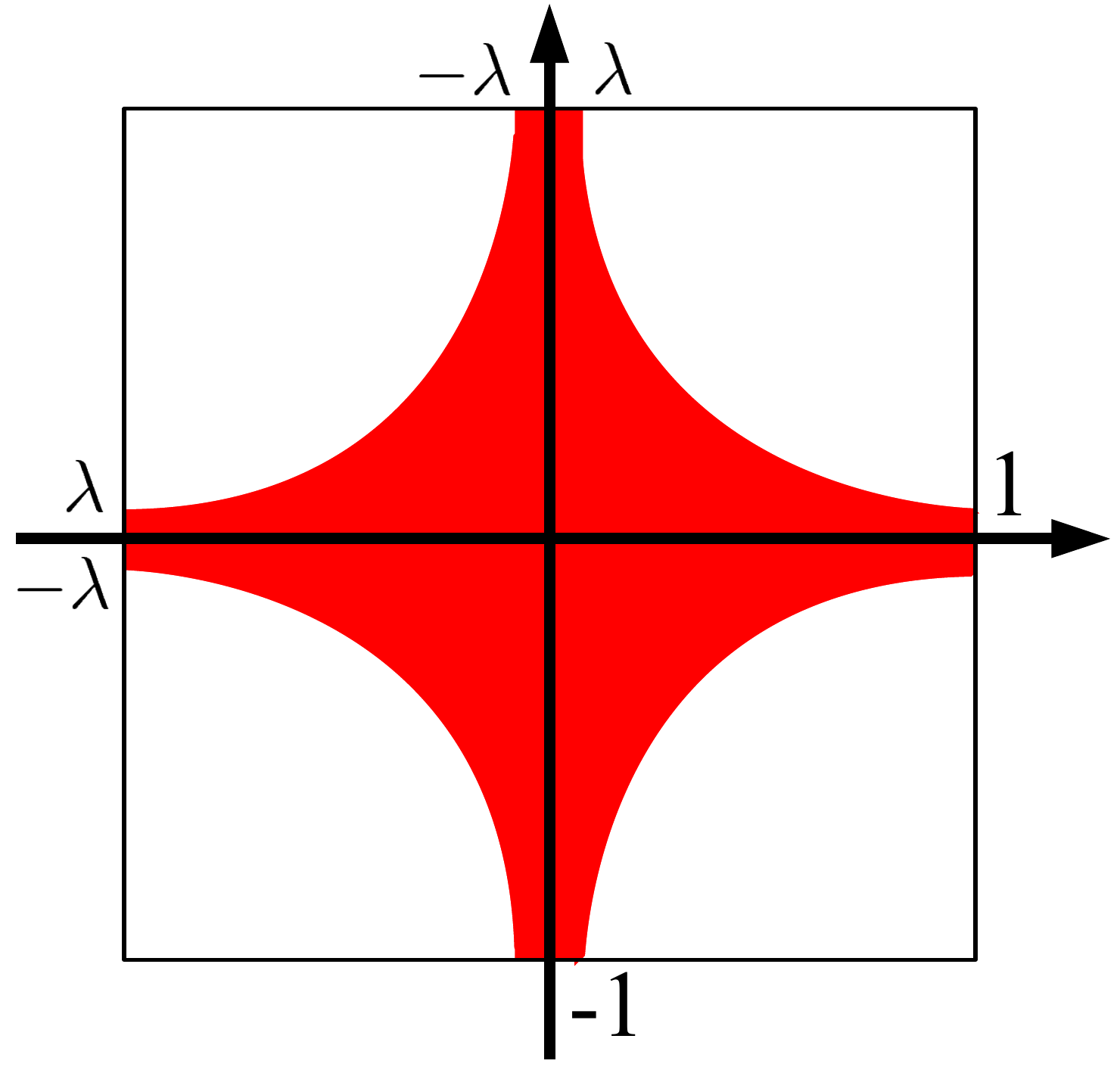}\label{fig:tubes_2}}\,
\subfigure[$f(x,y)=x^2y^3$]{\includegraphics[scale=0.245]{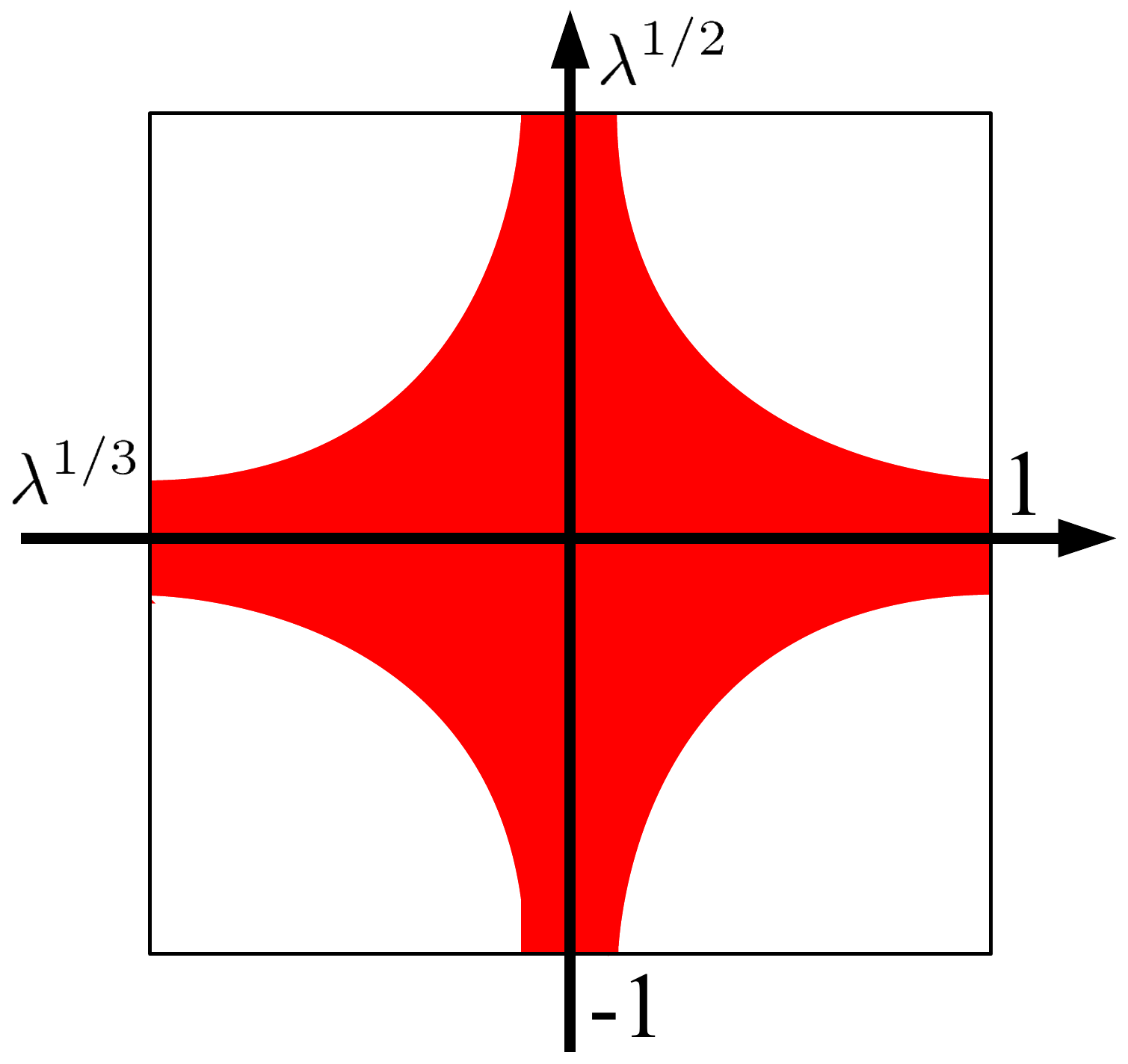}\label{fig:tubes_3}}\,
\subfigure[$f(x,y)=x^3y-xy^3$]{\includegraphics[scale=0.245]{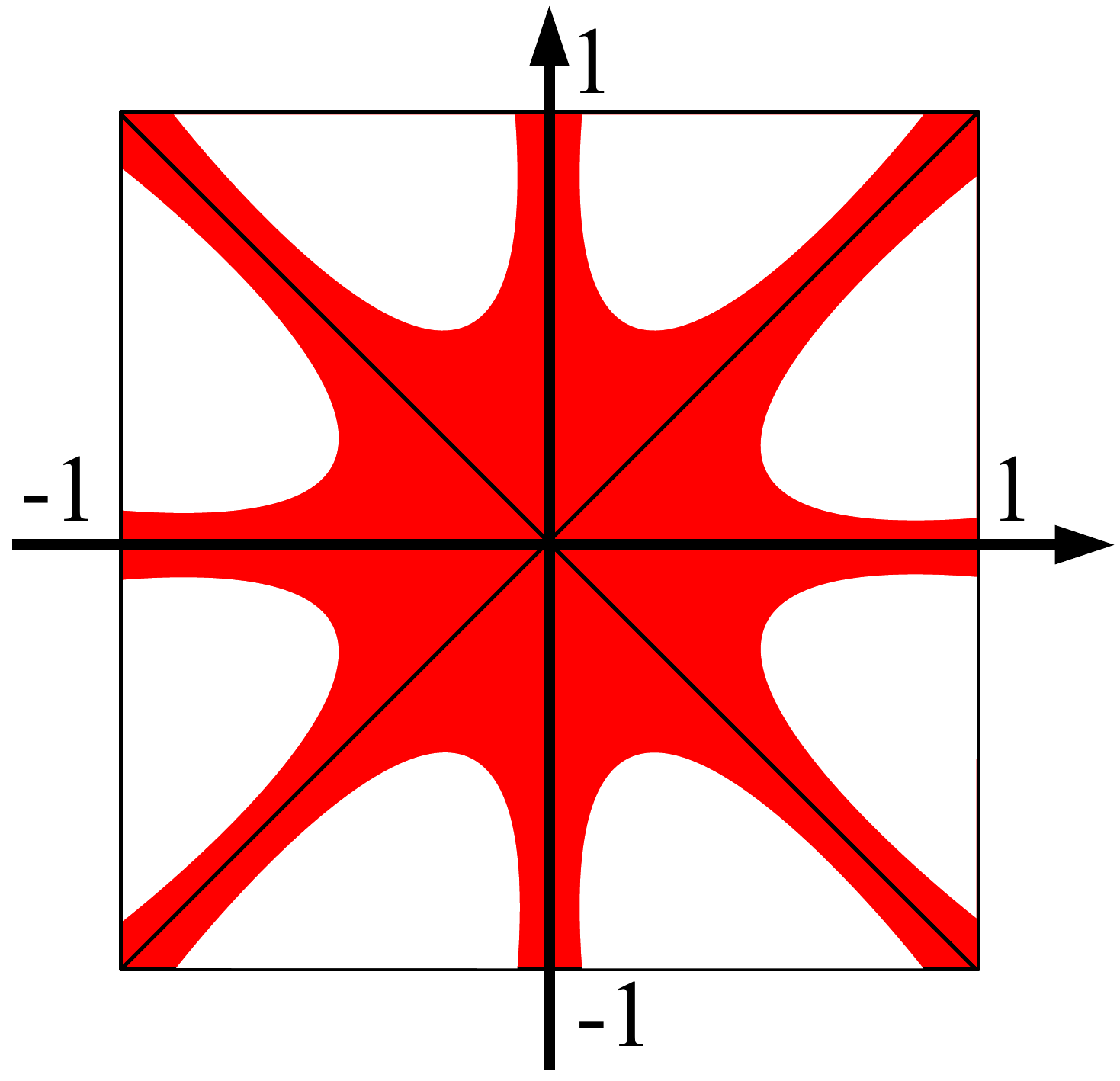}\label{fig:tubes_4}}
\caption{Tubes for various polynomials in two variables.}
\label{fig:tubes}
\end{figure}

\begin{ex}
\label{ex_tubes}
Let $d=2$ and $\varphi$ the Lebesgue probability measure on
the square $\,  \Omega = [-1,+1]^2 $.
Our problem is to compute the area of the tube
$\{(x,y)\in\Omega:|f(x,y)| \leq \lambda\}$. 
Here $f(x,y)$ is one of the four simple polynomials below
whose tubes are shown in Figure \ref{fig:tubes}.

\begin{enumerate}
\item[(a)] $f(x,y)=x$: 
The corresponding tube  is a rectangle and its area equals
$$V(\lambda) =  \lambda.$$
So, in this example, we have $(\ell,m)=(1,1)$ and $C=1$.
For other lines, the value of $C$ will change.
Proposition \ref{thm:smoothFunctionRLCT} below shows that 
 $(\ell,m)=(1,1)$ for smooth hypersurfaces.
\item[(b)] $f(x,y)=xy$: The tube in Figure \ref{fig:tubes_2}
consists of four copies of a region that is the union of a small rectangle
and a certain area under a hyperbola. Using calculus, we find
$$V(\lambda) \,\,=\,\, 4\left( \lambda + \int_{\lambda}^1 \frac{\lambda}{x} dx\right)\frac{1}{4} 
\,\,=\,\, \lambda(-\ln\lambda)+\lambda . $$
The logarithm function appears in this case.
 We have $(\ell,m)=(1,2)$ and $C=1$.
\item[(c)] $f(x,y)=x^2y^3$: The corresponding tube is shown in Figure \ref{fig:tubes_3}. Its area equals
$$V(\lambda) \,\,=\,\, 4\left( \lambda^{1/2} + \int_{\lambda^{1/2}}^1 \lambda^{1/3}x^{-2/3} dx\right)\frac{1}{4} 
\,\,=\,\,  3\lambda^{1/3}-2\lambda^{1/2}.$$
So, the real log canonical threshold equals
  $(\ell,m)=(\frac{1}{3},1)$, and we have $C=3$.
See Proposition \ref{prop:RLCTMonomial} for
   a  formula for $(\ell,m)$ when $f$ is a monomial in any number of variables.
\item[(d)] $f(x,y)=xy(x+y)(x-y)$: The corresponding tube is shown in Figure \ref{fig:tubes_4}. This example is a slight generalization of (b). As in (b) there is just one singularity at the origin,
given by the intersection of lines. Computing the area
$V(\lambda)$ is more challenging.
In Example \ref{ex:fourlines} we shall see that
the real log canonical threshold equals   $(\ell,m)=(\frac{1}{2},1)$.
\end{enumerate}
For general bivariate polynomials $f(x,y)$ we are facing a hard calculus problem, namely integrating the function  $y = y(x)$ that is defined implicitly by $f(x,y) = \lambda$.
We can approach this by expanding $y$ as a 
Puiseux series in $\lambda$ whose coefficients depend on $x$.
Integrating these coefficients leads to asymptotic formulas in $\lambda$.
These are consistent with what is to follow.
\qed
\end{ex}

We now return to the general setting defined at the beginning of this section.
Let $W$ be a random variable taking values in $\Omega$ with distribution $\varphi$. The volume $V(\lambda)$ with respect to the measure $\varphi$ can then be viewed as the 
cumulative distribution function of the random variable $f(W)$. The corresponding 
probability distribution
function $
v(\lambda) = dV/d\lambda
$
is called the \emph{state density function}. Its Mellin transform is known 
as the  \emph{zeta function} of $f$. It is denoted by
$$
    \zeta(z) \,=\, \int_0^{\infty} \lambda^{-z} v(\lambda) d\lambda \,=\, \int_\Omega
|f(\omega)|^{-z}\varphi(\omega) d\omega 
\qquad \hbox{for }z \in \C.
$$
According to asymptotic theory
\cite{AGV,LinThesis,Watanabe09}, our volume  has
the asymptotic series expansion
\begin{equation}
\label{eq:asymp}
V(\lambda) \,\approx\, \sum_{\ell} \sum_{m=1}^d C_{\ell,m} \lambda^\ell (-\ln \lambda)^{m-1}.
\end{equation}
Here the index $\ell$ runs over some arithmetic progression of positive
rational numbers and $d$ is the dimension of the parameter space
$\Omega$. The equation (\ref{eq:asymp}) is valid
for sufficiently small $\lambda > 0$.
To be precise, writing $V(\lambda) \approx \sum_{i=1}^{\infty}
g_i(\lambda)$, where $g_1(\lambda) > g_2(\lambda) > \cdots$ for small
$\lambda$, means that 
\begin{equation}
\label{eq:curly_equal}
\lim_{\lambda \rightarrow 0} \frac{V(\lambda) - \sum_{i=1}^k
  g_i(\lambda)}{g_k(\lambda)} \,=\, 0 \qquad
\hbox{for each positive integer $k$}.
\end{equation}
Using the little-o notation, this is equivalent to $V(\lambda) = \sum_{i=1}^{k}
g_i(\lambda) + o(g_k(\lambda))$ as $\lambda \rightarrow 0$ for each positive integer $k$. It is a common
misconception to think that the infinite series converges to $V(\lambda)$
for each fixed $\lambda$ when $\lambda$ is small. Rather, it means that for each fixed $k$, the $k$-term approximation for
$V(\lambda)$ gets better as $\lambda \rightarrow 0$. We will primarily be interested in the first term approximation
(\ref{eq:firstorder}).

\begin{defn}[{\cite[\S 4.1]{LinThesis}},{\cite[\S
    7.1]{Watanabe09}}] \label{thm:DefRLCT}
    We here define the {\em real log canonical
threshold} $(\ell, m)$ of $f$ over $\Omega$ with respect to $\varphi$.
This is a pair in $\mathbb{Q}_+\times\mathbb{Z}_+$ which we denote by
$\RLCT_\Omega(f; \varphi)$. It  measures
 the complexity of the singularities of the hypersurface defined by
$f(\omega)=0$. 

The following four definitions 
of $\,\RLCT_\Omega(f; \varphi) = (\ell,m)\,$ are known to be equivalent:
\begin{enumerate}
\item[(i)] For large $N>0$, the {\em Laplace integral}
$$
\displaystyle Z(N) = \int_\Omega e^{-N|f(\omega)|}\, \varphi(\omega)\,d\omega
$$ 
is asymptotically $C N^{-\ell} (\ln N)^{m-1}$ for some
constant $C$.
\item[(ii)] The {\em zeta function}
$$
\displaystyle \zeta(z) = \int_\Omega | f(\omega)|^{-z} \, \varphi(\omega)\,d\omega
$$
has its smallest pole at $z = \ell$ and that pole has multiplicity $m$.
\item[(iii)] For small $\lambda > 0$, the {\em volume function}
$$
V(\lambda) = \int_{|f(\omega)|\leq \lambda}\, \varphi(\omega)\,d\omega
$$
is asymptotically $C\,\lambda^{\ell} (-\ln \lambda)^{m-1}$ for some
constant $C$.
\item[(iv)] For small $\lambda > 0$, the {\em state density function}
$$
v(\lambda) = \frac{d}{d\lambda} \int_{|f(\omega)|\leq \lambda}\, \varphi(\omega)\,d\omega
$$
is asymptotically $C \lambda^{\ell-1}(-\ln \lambda)^{m-1}$ for some
constant $C$.
\end{enumerate}
If the real analytic hypersurface $\{\omega \in \Omega : f(\omega)=0\}$ is empty, we
set $\ell = \infty$ and we leave $m$ undefined. We say
  that $(\ell_1, m_1) <
(\ell_2,m_2)$ if $\ell_1 < \ell_2$ or if $\ell_1 = \ell_2$ and $m_1 >
m_2$. Hence, the pairs are ordered reversely by the size of $\lambda^{\ell} (-\ln \lambda)^{m-1}$
for sufficiently small $\lambda > 0$.
\end{defn}

Let us provide some intuition for the ordering of the pairs $(\ell,m)$.
The real log canonical threshold is a measure of complexity for singularities.
Analytic varieties can be
stratified into subsets where this measure is constant.
The highest stratum contains
the smooth points of the variety. As we go deeper, to strata with
lower real log canonical thresholds, we
encounter singularities of increasing complexity. The volumes of $\lambda$-fattenings of deeper
singularities will, asymptotically as $\lambda$ goes to zero, also be larger than those of their less complex
counterparts. For instance, in both Figures \ref{fig:tubes_2} and
\ref{fig:tubes_3} the singular locus consists of the origin, but  the $\lambda$-fattening of the origin in  Figure \ref{fig:tubes_3} is larger than in Figure \ref{fig:tubes_2}. 
See also Example \ref{ex:mono_s_t}.

\begin{ex}
\label{ex:d-ball}
Let $f(\omega) = \omega_1^2 + \omega_2^2 + \cdots  +\omega_d^2$
and $\varphi$ the Lebesgue probability measure on $\Omega = [-1,+1]^d$.
Then ${\rm Tube}(\lambda)$ is the {\em standard ball} of radius $\lambda^{1/2}$, whose
$\varphi$-volume is
\begin{equation*}
 V(\lambda) \,\, = \,\,
\frac{\pi^{d/2}}{2^d \cdot \Gamma(\frac{d}{2} + 1)} \cdot \lambda^{d/2}.
\end{equation*}
By Definition \ref{thm:DefRLCT} (iii),  the  real log canonical threshold equals
$\,{\rm RLCT}_\Omega(f; \varphi) = (d/2,1)$. \qed
\end{ex}

We now list some formulas for computing the real log canonical threshold. A first useful fact is that
$\RLCT_\Omega(f; \varphi)$  is independent of
       the underlying measure $\varphi$ as long as it is positive everywhere. 
We can thus assume that   $\varphi$ is the uniform distribution on~$\Omega$.

\begin{prop}
\label{prop:phi}
 If $\varphi:\Omega \rightarrow \R$ is strictly positive and
$1$ denotes the constant unit function on $\Omega$, then
 $$\RLCT_\Omega(f;\varphi)\, = \,\RLCT_\Omega(f;1).$$
 \end{prop}
\begin{proof} See \cite[Lemma 3.8]{LinThesis}.
\end{proof}

\begin{prop} \label{prop:RLCTMonomial}
Suppose that $\Omega$ is a neighborhood of the origin. If $\,f(\omega)=
\omega_1^{\kappa_1} \cdots
\omega_d^{\kappa_d} g(\omega)\,$ 
where the $\kappa_i$ are nonnegative integers and
the function $g:\Omega \rightarrow \R$ does not have any zeros, then
$\,\RLCT_\Omega(f;1) = (\ell,m)\,$ where 
$$\ell \, =\, \min_i \frac{1}{\kappa_i} \quad \hbox{and} \,\,\quad m
\,=\,  \left|\left\{\operatorname*{argmin}_i \frac{1}{\kappa_i} \right\}\right|.$$
\end{prop}
\begin{proof}
This is a special case of Theorem \ref{thm:RLCTMonomial2}
which will be proved later.
\end{proof}

Recall that
an analytic hypersurface $\{f(\omega) = 0\}$ is \emph{singular} at
a point $\omega \in \Omega$ if $\omega$ satisfies 
$$
f(\omega) = 0 \quad \hbox{and} \,\,\quad \frac{\partial f}{\partial \omega_i}(\omega) \,=\, 0 \quad \text{for }i =1, \ldots, d.
$$
If the hypersurface is not singular at any point $\omega \in \Omega$, then it is said to be
 $\emph{smooth}$.

\begin{prop} \label{thm:smoothFunctionRLCT}
If the hypersurface $\{f(\omega) = 0$\} is smooth 
then $\,\RLCT_\Omega(f;1) = (1,1)$.
\end{prop}

\begin{proof}
This is also a  special case of Theorem \ref{thm:RLCTMonomial2}.
\end{proof}

\begin{ex} \label{ex:mono_s_t}
Following up on
Example \ref{ex_tubes}, we now consider an arbitrary
monomial function $f(x,y)=x^sy^t$ on the square
$\Omega = [-1,1]^2$.
The tube looks as in Figure \ref{fig:tubes_3}.
Its area satisfies
  $$V(\lambda)
\approx \
\left\{\begin{array}{ll} 
C\lambda^{1/s} & \textrm{if } s<t \\
C\lambda^{1/t} & \textrm{if } s > t \\
C\lambda^{1/s}(-\ln \lambda) & \textrm{if } s=t
\end{array}\right.$$
This formula for the asymptotics (\ref{eq:firstorder}) 
follows from 
Definition \ref{thm:DefRLCT} (iii)
and Proposition \ref{prop:RLCTMonomial}.
\qed
\end{ex}

For the statistical applications  in this paper, the relevant  functions $f$ are polynomials.
They are determinants $ f = \det (K_{iR,jR})$, where $R = V\backslash (S\union \{i,j\})$
as in Section \ref{sec:intro}.
 Let $\RLCT(i,j|S)$ denote the corresponding real log canonical threshold over 
 $\Omega = [-1,1]^E$ with  respect to a positive density $\varphi$.
The theory developed so far says that 
the real log canonical threshold of the correlation hypersurface
gives an asymptotic volume formula for~$V_{i,j|S}(\lambda)$.

\begin{thm} \label{lemma_unfaithful}
If $\,\varphi\,$ satisfies the assumptions in Proposition \ref{prop:phi}, then as $\lambda$ tends to zero, the volume of the region ${\rm
  Tube}_{i,j|S}(\lambda)$ (see (\ref{eq:tube_ijS})) is asymptotically
$$
V_{i,j|S}(\lambda) \,\approx\,  C\,\lambda^{\ell} (-\ln \lambda)^{m-1}
$$
for some constant $C>0$ (which only depends on $G$) and $(\ell,m) = \RLCT(i,j|S)$.
\end{thm}

\begin{proof}
By part (iii) in Definition \ref{thm:DefRLCT},
the desired pair $(\ell, m)$ is the real log
canonical threshold of the partial correlation
$f = \corr(i,j|S)$. This is the algebraic (and hence analytic) function  in (\ref{eq:parcor}).
This function differs from the polynomial
$\det(K_{iR,jR})$ by a denominator that does not vanish over
$\Omega$. That denominator is a unit in the ring of real analytic
functions over $\Omega$, and multiplying by a unit does~not change the
RLCT of an analytic function
\cite[\S 4.1]{LinThesis}.
\end{proof}

We close this section by relating our results directly to the
study of unfaithfulness in \cite{URBY}.

\begin{cor}\label{prop_unfaithful}
Under the assumptions in Theorem \ref{lemma_unfaithful}, as $\lambda$ tends to zero, the volume of $\lambda$-strong-unfaithful distributions  satisfies
$$
V_G(\lambda) \, \approx\, C\,\lambda^{\ell} (-\ln \lambda)^{m-1}
$$
for some constant $C > 0$. Here $(\ell,m)$ is the minimum of the pairs
$\RLCT(i,j|S)$, where  $(i,j,S)$ runs over all triples
 in the DAG $\,G$ such that $i$ is not d-separated from $j$ given $S$.
\end{cor}
\begin{proof}
The function $V_G(\lambda)$ is the volume of the union of the regions
${\rm Tube}_{i,j|S}(\lambda)$. Thus,
$$
\max_{i,j,S} V_{i,j|S}(\lambda) \,\leq\, V_G(\lambda) \,\leq\, \sum_{i,j,S} V_{i,j|S}(\lambda)
$$
Asymptotically, for small positive values of $\lambda$, both the lower and upper bounds
vary like a constant multiple of $\lambda^{\ell}(-\ln \lambda)^{m-1}$
where  $(\ell,m)$ is the minimum over all pairs $\RLCT(i,j|S)$. In this minimum,
 $(i,j,S)$ runs over all triples such that
  $i$ and $j$ are d-connected given $S$.
\end{proof}

\section{Singular Locus}
\label{sec:sing_locus}

The asymptotic integration theory in Section \ref{sec:real_log_threshold} 
requires us to analyze the singular locus ${\rm Sing}(f)$ of the
real algebraic hypersurface determined by a given polynomial $f$.
If ${\rm Sing}(f)$ is empty then the hypersurface is smooth and
Proposition \ref{thm:smoothFunctionRLCT}
characterizes the asymptotics of the integral.
In this section we return to Gaussian graphical models,
we develop tools for computing the relevant singular loci,
and we show that they are empty in many cases.
In many of the remaining cases, the singularities are
of the monomial type featured in Proposition~\ref{prop:RLCTMonomial}.

Consider any almost-principal minor
$\,f = {\rm det}(K_{iR, jR})\,$ of the concentration matrix
$K$ of a DAG $G$. This is a polynomial function on the parameter space
$\mathbb{R}^E$. This polynomial and its partial derivatives
are elements in the polynomial ring $\mathbb{Q}[a_{ij}: (i,j) \in E]$.
The {\em Jacobian ideal} of $f$ 
 is the ideal in this polynomial ring generated
 by $f$ and its partials. We denote it by
$$\, {\rm Jacob}_{i,j,R} \,\,\, := \,\,\,
\langle \,f \,\rangle \,+ \,
\bigl\langle \,\frac{\partial f }{\partial a_{ij}} \,: \,(i,j) \in E \,
\bigr\rangle . $$
The singular locus ${\rm Sing}(f)$ is the
subvariety of real affine space $\R^E$ 
defined by  the Jacobian ideal ${\rm Jacob}_{i,j,R}$.
The structure of the real variety ${\rm Sing}(f)$ governs
the volume $V_{i,j|S}(\lambda)$ of the set
${\rm Tube}_{i,j|S}(\lambda)$ of unfaithful parameters.
If ${\rm Sing}(f) = \emptyset$ then
Proposition \ref{thm:smoothFunctionRLCT} tells us that $V_{i,j|S}(\lambda)$
asymptotically equals $C\lambda$ for some constant $C >0$.
If the singular locus is not empty then understanding ${\rm Sing}(f)$
is essential for computing its real log
canonical threshold $(\ell,m)$.

We conducted a comprehensive study of all DAGs with few nodes
by computing the singular locus for every
almost-principal minor in their concentration matrix $K$.
Our first result concerns the special case of complete graphs.
Non-complete graphs will be studied~later.

\begin{thm} \label{thm:uptosix}
Suppose that $\,\varphi\,$ satisfies the assumptions in Proposition \ref{prop:phi}.
For any conditional independence statement 
on the complete directed graph $K_p$ with $p \leq 6$ nodes,
we have ${\rm Sing}(f) = \emptyset$, and hence
$V_{i,j|S}(\lambda) \approx C \lambda$ for all triples $(i,j,S)$.
\end{thm}

It is tempting to conjecture that the hypothesis $p \leq 6$ can be 
removed in this theorem. Presently we do not know how
to approach this problem other than by direct calculation.

Applying Corollary \ref{prop_unfaithful}, this means that the volume of $\lambda$-strong-unfaithful distributions for the complete graph satisfies
$V_{K_p}(\lambda) \, \approx\, C\,\lambda$ for $\lambda\to 0$, which is the best possible behavior regarding strong-faithfulness. This may be counter-intuitive, but is confirmed in simulations. In Figure \ref{fig:all_classes} we plot (via (\ref{funct_plots})) the proportion of strong-unfaithful distributions $V_G(\lambda)$ for the five graphs in Section \ref{sec:graphs} for varying values of $\lambda$. Especially in the plot for $p=10$ it becomes apparent that the behavior for $\lambda\to 0$ is very different than, say, for $\lambda=0.001$. For $\lambda\to 0$ we 
have $V_{\rm complete}(\lambda)<V_{\rm chain}(\lambda)$, although the chain-like graph is much sparser than the complete graph.
Note also that the complete graph $K_{10}$ has
 $\sum_{k=2}^{10}\binom{10}{k}\binom{k}{2}=11520$ relevant triples $(i,j,S)$,
  whereas for ${\rm Chain}_{10}$
 there are only $\sum_{k=1}^{9}k2^{k-1} =4097$ such triples.
 

\begin{figure}[t!]
\centering
\subfigure[$p=6$]{\includegraphics[scale=0.42]{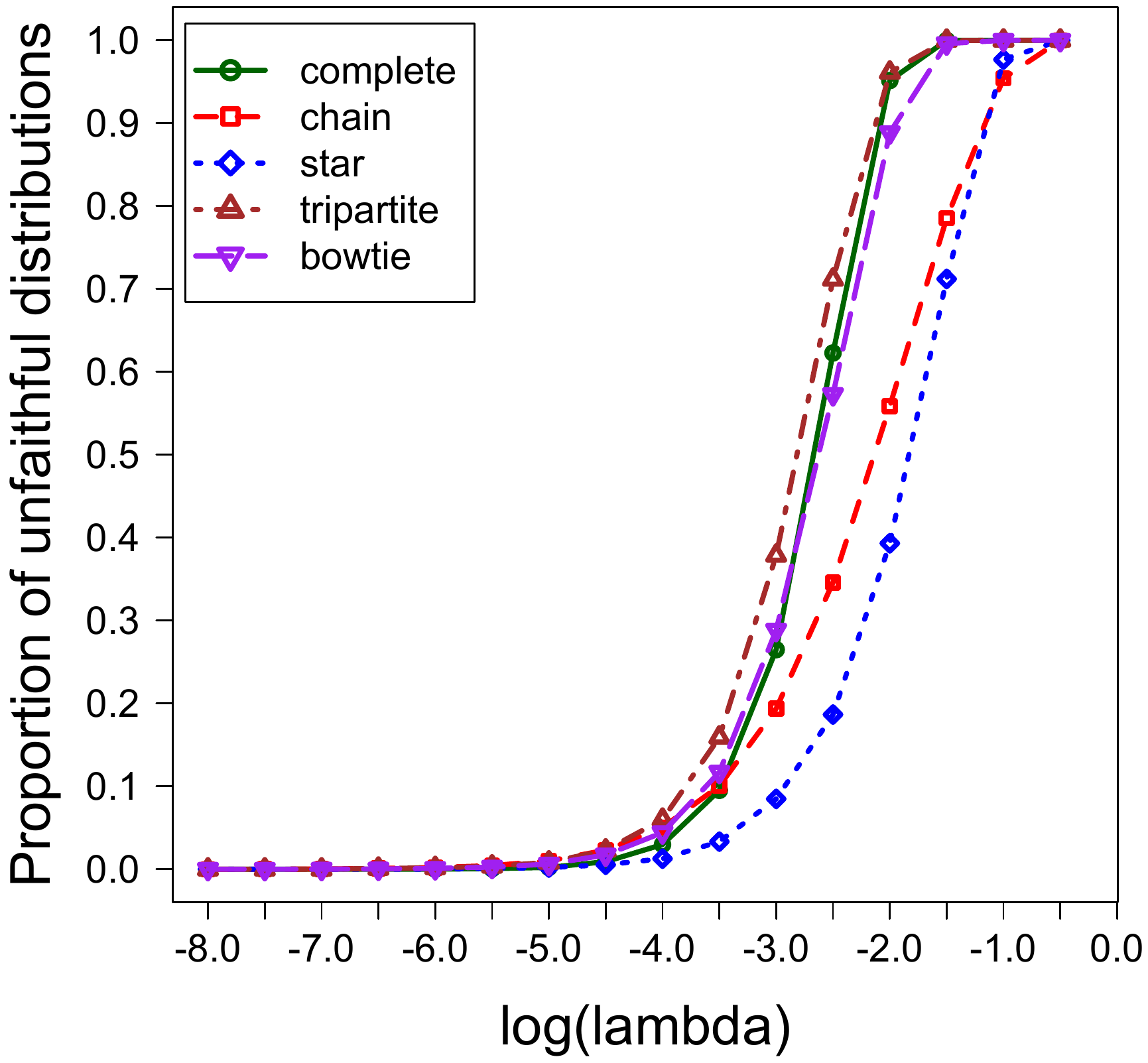}\label{Rplot_6}}\qquad
\subfigure[$p=10$]{\includegraphics[scale=0.42]{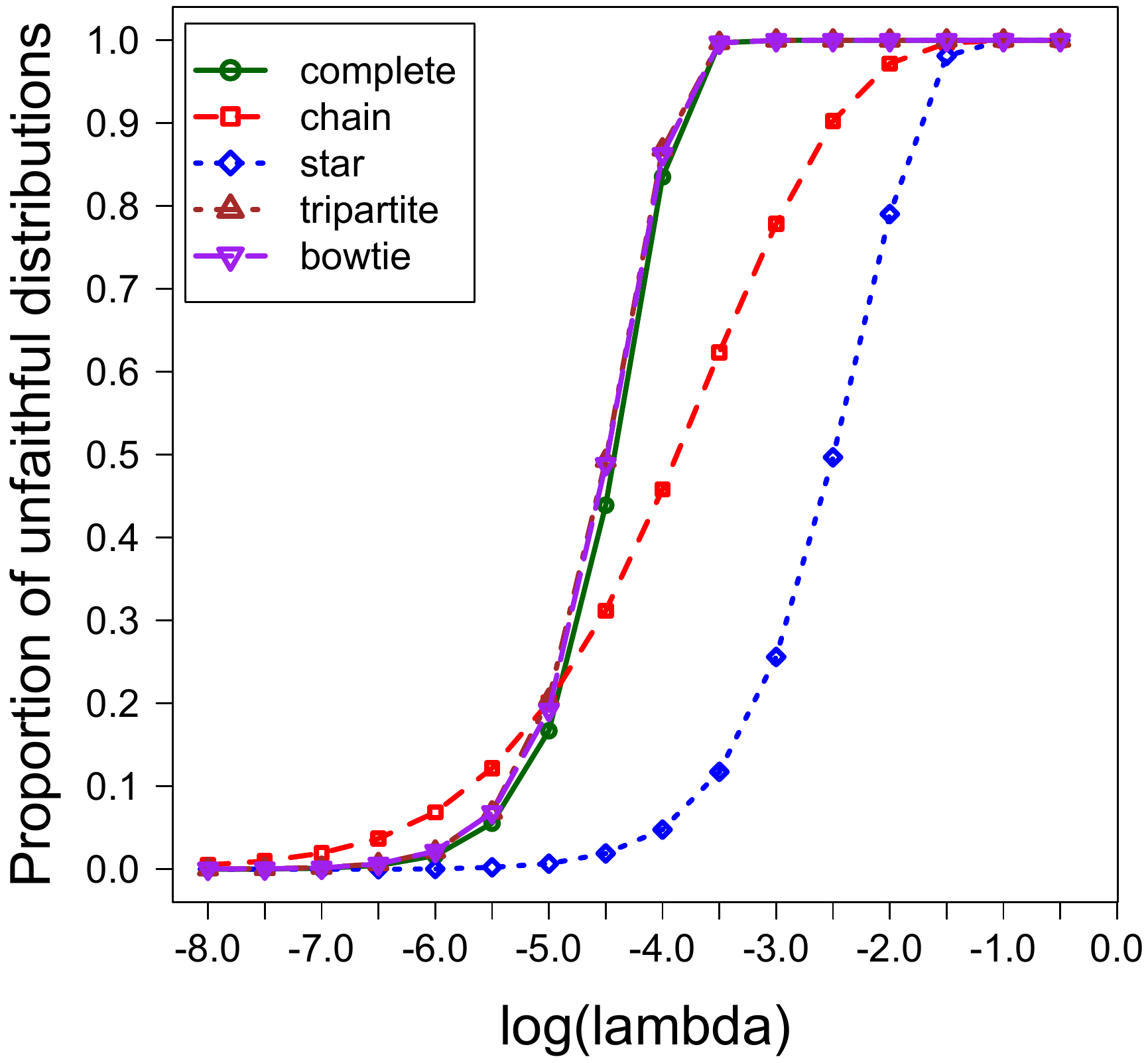}\label{Rplot_10}}
\caption{$V_G(\lambda)$ for the complete graph $K_p$ compared to $\textrm{Chain}_p$,  
$\textrm{Star}_p$, $\textrm{Tripart}_{p,2}$, $\textrm{Bow}_p$.}
\label{fig:all_classes}
\end{figure}

In what follows we explain the algebraic computations that 
led to Theorem \ref{thm:uptosix}. We used ideal-theoretic
methods from \cite{CLO} in their implementation in the
Gr\"obner-based software packages {\tt Macaulay 2} 
\cite{M2} and {\tt Singular} \cite{Singular}.
An important point to note at the outset is that the ideal
${\rm Jacob}_{i,j,R}$ is almost never the unit ideal.
By Hilbert's Nullstellensatz, this means that the
hypersurfaces $\,f = {\rm det}(K_{iR, jR}) \,$ do have
plenty of singular points over the field $\mathbb{C}$
of complex numbers. What Theorem 
\ref{thm:uptosix} asserts is that, in many of the cases of interest
to us here, none of those singular points have their
coordinates in the field $\mathbb{R}$ of real numbers.

In order to study the real variety of an ideal, techniques
from real algebraic geometry are needed. A key technique
is to identify sums of squares (SOS). Indeed, the {\em real
Nullstellensatz} \cite{Marshall} states that the real variety is empty if and only if
the given ideal contains a certain type of SOS.
To apply this to directed Gaussian graphical models, we shall
use the fact that every principal minor of
the covariance matrix or the concentration matrix furnishes such an SOS.

\begin{lem}
\label{lemma:eins}
Every principal minor ${\rm det}(K_{R,R})$
of the concentration matrix $K$ of a DAG 
is equal to $1$ plus a sum of squares in $\mathbb{Q}[a_{ij}: (i,j) \in E]$.
In particular, its real variety is empty.
\end{lem}

\begin{proof}
We can write the principal submatrix $K_{R,R}$ as the product
$\,(A-I)_{R,*} \cdot ((A-I)_{R,*})^T$,
where $( \,\,)_{R,*}$ refers to the submatrix with row indices
$R$. Thus  $K_{R,R}$ is the product of
an $|R| \times p$-matrix and its transpose. By the Cauchy-Binet formula,
${\rm det}(K_{R,R})$ equals the sum of squares of all maximal minors of the 
$|R| \times p$-matrix
 $\,(A-I)_{R,*} $. One of these maximal minors is the identity matrix. Hence the polynomial
${\rm det}(K_{R,R})$ has the form $1 + {\rm SOS}$. In particular, 
the matrix $K_{R,R}$ is invertible for all parameter values in   $\mathbb{R}^E$.
\end{proof}

We note that Lemma \ref{lemma:eins} holds more generally also in the case of unequal noise variances.
 In the context of commutative algebra, it now makes sense to introduce
the {\em saturations}
\begin{equation*}
 \begin{matrix}
{\rm Singu}_{i,j,R} & = & \bigr(\,{\rm Jacob}_{i,j,R} \,:\, {\rm det}(K_{R,R})^\infty \,\bigr) ,\quad  \\
{\rm Singu}_{i,j,R}^* & = & \,\bigr(\,{\rm Singu}_{i,j,R} \,:\, (\prod_{(i,j) \in E} a_{ij} )^\infty \,\bigr) .
\end{matrix}
\end{equation*}
These are also ideals in $\mathbb{Q}[a_{ij} : (i,j) \in E]$.
By definition,  ${\rm Singu}_{i,j,R} $ consists of all polynomials that get multiplied into
the Jacobian ideal by some power of the determinant of $K_{R,R}$,
and ${\rm Singu}_{i,j,R}^*$ consists of polynomials that get multiplied
into  ${\rm Singu}_{i,j,R} $ by some monomial.
By \cite[\S 4.4]{CLO}, the variety of ${\rm Singu}_{i,j,R}$ is the Zariski closure of
the set-theoretic difference of the variety of $\,{\rm Jacob}_{i,j,R} \,$
and the hypersurface  $\,\{{\rm det}(K_{R,R}) = 0 \}$. 
We saw in Lemma \ref{lemma:eins} that the latter hypersurface
has no real points.
The ideal ${\rm Singu}^*_{i,j,R}$ represents singularities in $(\R\backslash \{0\})^E$.

\begin{cor}
The singular locus of the real algebraic hypersurface 
 $\bigl\{{\rm det}(K_{iR,jR}) = 0 \bigr\} $ in $\mathbb{R}^E$ 
coincides with the set of real zeros of the ideal $\, {\rm Singu}_{i,j,R}$.
The set  of real zeros of $\, {\rm Singu}^*_{i,j,R}$ is the Zariski
closure of the subset of all singular points whose coordinates are non-zero.
\end{cor}

\begin{proof}[Proof of Theorem \ref{thm:uptosix}]
We computed the ideals ${\rm Jacob}_{i,j,R}$
and ${\rm Singu}_{i,j,R}$ for every
almost-principal minor $K_{iR,jR}$ in the concentration matrices
of the graphs $G = K_3, K_4, K_5, K_6$.
In all cases the ideal ${\rm Singu}_{i,j,R}$
was found to equal the unit ideal $\langle \,1 \,\rangle$.
These exhaustive computations 
were carried out using 
the software {\tt Singular} \cite{Singular}.
This establishes Theorem~\ref{thm:uptosix}.
\end{proof}

We briefly discuss our computations for the complete directed graph on six nodes.

\begin{ex}
\label{ex:K6}
Fix the complete directed graph $G = K_6$.
We tested all 240 conditional independence statements and computed the corresponding ideal ${\rm Singu}_{i,j,R}$. We discuss one interesting instance, namely $i= 1, j = 3,  R  =\{2,4\}$.
The almost-principal minor $K_{241,243} = $
$$ \begin{pmatrix}
 a_{23}^2+a_{24}^2+a_{25}^2+a_{26}^2+1 &
 a_{25} a_{45} + a_{26} a_{46} - a_{24} & 
  a_{24} a_{34} {+} a_{25} a_{35} {+} a_{26} a_{36} {-} a_{23} \\
  a_{25} a_{45} + a_{26} a_{46} - a_{24} &
   a_{45}^2+a_{46}^2+1 &  a_{35} a_{45} + a_{36}  a_{46} - a_{34} \\
a_{13} a_{23} {+} a_{14} a_{24} {+} a_{15} a_{25} {+} a_{16} a_{26} {-} a_{12} &
 a_{15} a_{45} + a_{16} a_{46} - a_{14} &
  a_{14} a_{34} + a_{15} a_{35} + a_{16} a_{36} \end{pmatrix}
  $$
  contains all $15$ parameters except $a_{56}$.
  Its determinant is a polynomial of degree $6$.
  Of  its $14$ partial derivatives, $13$ have degree $5$.
  The derivative with respect to $a_{12}$ has
    degree $4$. Thus 
     ${\rm Jacob}_{1,3,\{2,4\}}$ is generated by
    $15$ polynomials of degrees $4,5,\ldots,5,6$.
  The matrix $K_{24,24}$ is the upper left
     $2 {\times} 2$-block in
  the matrix above. The square of its determinant
  is a polynomial of degree $8$ that happens to lie 
     in the ideal ${\rm Jacob}_{1,3,\{2,4\}}$.
     This proves ${\rm Singu}_{1,3,\{2,4\}} = \langle 1 \rangle$.
\qed     
\end{ex}

For graphs $G$ that are not complete, ${\rm Singu}_{i,j,R}$ may not be the unit ideal.
We already saw one non-obvious instance of this for
the tripartite graph in Example~\ref{ex:tripartite6}.
Here is an even smaller example where
the Jacobian ideal and its saturations are equal,
and not the unit ideal.

\begin{ex}
\label{ex:sec4graph1}
 Let $p = 4$ and take $G$ to be the almost-complete
graph with adjacency matrix
$$ A_G \quad = \quad
\begin{pmatrix}
 0 &  0 & a_{13} &  a_{14} \\
  0 & 0 & a_{23} &  a_{24} \\
   0 & 0 & 0  &  a_{34} \\
   0 & 0 & 0 & 0   
   \end{pmatrix}.
$$
The conditional independence statement $1 \independent 2 \given 4$
is represented by the almost-principal minor
$$ K_{31,32} \quad = \quad 
\begin{pmatrix}
   a_{34}^2+1  &  a_{24} a_{34}-a_{23} \\
    a_{14} a_{34} -a_{13} &  a_{13} a_{23} + a_{14} a_{24}
  \end{pmatrix}
 $$
 of the concentration matrix.
 The determinant of this minor
  factors into two binomial factors:
\begin{equation}
\label{eq:itfactors}
 {\rm det}(K_{31,32}) \quad = \quad 
(a_{13} a_{34} + a_{14}) (a_{23} a_{34} + a_{24} ).
\end{equation}
The Jacobian ideal is the prime ideal generated by these factors:
$$ {\rm Jacob}_{1,2,3} \,=\, {\rm Singu}_{1,2,3} \,\,=\,\,
{\rm Singu}_{1,2,3}^* \,\,=\,\,
\langle \,a_{13} a_{34} + a_{14} \,, \,a_{23} a_{34} + a_{24} \,\rangle . $$
The left equality holds because 
$\,{\rm det}(K_{3,3}) = a_{34}^2+1\,$ is a non-zerodivisor modulo
$ {\rm Jacob}_{1,2,3} $.
The singular locus of (\ref{eq:itfactors}) is the three-dimensional real variety
   defined by this binomial ideal
 in the parameter space $\mathbb{R}^5$.
Its real log canonical threshold is found to be $(\ell,m) = (1,2)$. \qed
 \end{ex}
 
 \begin{table}[b]
\caption{RLCT for all DAGs with three nodes}
\begin{center}
\begin{tabular}{lccc}
\hline
             & $(1,1)$ & $(1,2)$ & {\em Subtotal}\\
\hline
Monomial &21& 3& {\em 24} \\
Smooth & 3& &  {\em 3} \\
\hline
 {\em Subtotal} & {\em 24} & {\em 3}& {\em 27}\\
\hline
\end{tabular}
\label{table1}
\end{center}
\end{table}

\begin{table}[!t]
\caption{RLCT for all DAGs with four nodes}
\begin{center}
\begin{tabular}{lccccc}
\hline
              & $(1,1)$ & $(1,2)$ & $(1,3)$ & $(1/2,1)$ &{\em Subtotal}\\
\hline
Monomial &568 & 145& 14& 1& {\em  728} \\
Smooth & 198& & & & {\em  198} \\
Normal crossing& & 22 & 2 & & {\em  24}\\
Blowup &12 & & & & {\em  12} \\
Special &2 & 1& & & {\em  3}\\
\hline
{\em  Subtotal} & {\em  780} & {\em  168} & {\em  16} & {\em  1} & {\em  965}\\
 \hline
\end{tabular}
\label{table2}
\end{center}
\end{table}

This example inspired us to analyze the partial correlations of all small DAGs with 
$p \leq 4$ nodes. In our
experiments, we found that $\det(K_{iR,jR}) $ is frequently
 the product of a monomial with a strictly positive sum of squares. This is the case when there is a unique path which d-connects nodes $i$ and $j$ given $S$.
 For instance, this holds
     for trees.
   Such cases are denoted as ``Monomial'' in Tables \ref{table1} and \ref{table2}. For these,
 the RLCT is read off directly from Proposition
\ref{prop:RLCTMonomial}. 
The rows labeled ``Smooth'' cover cases that are not monomial but where ${\rm
  Singu}_{i,j,R}$ is the unit ideal, so
  Proposition \ref{thm:smoothFunctionRLCT} gives us the RLCT.
  The next theorem summarizes the complete results.
    The trivial case $p=2$ is excluded because
there is only one graph $1\rightarrow 2$, with $\RLCT(1,2|\emptyset) = (1,1)$. 
Here and in Tables 1 and 2 we enumerate unlabeled DAGs.

\begin{thm} \label{thm:smallDAGs}
 Under the assumptions in Theorem \ref{lemma_unfaithful}, for all  DAGs with $p \leq 4$ nodes
 and all triples $(i,j,S)$,
 the value $\RLCT(i,j|S)$ is
 given in Tables \ref{table1} and \ref{table2}.
 In all cases but one, we have
   $\RLCT(i,j|S)=(1,m)$ where $m < p$.
\end{thm}

To establish Theorem \ref{thm:smallDAGs}, we listed every DAG $G$ and every
triple $(i,j,S)$ that is not d-separated in $G$.
The rows ``Monomial'' and ``Smooth'' were discussed above. On three nodes there are only 3 partial correlations that correspond to the weighted sum of more than one d-connecting path, namely the partial correlations $\textrm{corr}(1,2\mid 3)$, $\textrm{corr}(1,3)$, $\textrm{corr}(2,3)$ in the complete DAG $1\to 2$, $2\to 3$, $1\to 3$. These are the 3 cases of smooth RLCTs in Table~1. The row ``Normal crossing'' refers to cases  covered by
Theorem \ref{thm:RLCTMonomial2}. The ``Special'' cases are treated in Examples
\ref{ex:sec4graph1} and \ref{ex:specialcases}. Lastly, the row
``Blowup'' represents instances where the real singular locus is a
linear space. Our computation of  ${\rm RLCT}(i,j|S) = (\ell,m)$ for such instances uses the method
 in Example \ref{ex:4charts}.
We now examine the unique exceptional case where $\ell \not= 1$.

\begin{ex}
\label{ex:sec4graph2}
 Let $p = 4$ and $G = {\rm Tripart}_{4,1}$. Its concentration matrix
 may be obtained from Example \ref{ex:sec4graph1}
  by setting $a_{14} = a_{24} = 0$.
  The partial correlation
 for $1 \independent 2 \given 4$ is now given by 
  $${\rm det}( K_{13,23})\, =\, a_{13} a_{23} a_{34}^2 . $$
For this monomial,  Proposition  \ref{prop:RLCTMonomial} tells us that
$\,(\ell,m) = {\rm RLCT}(1,2|4) \,= \,(1/2,1)$.
 \qed
 \end{ex}
 
Here is an interesting case where  the RLCT
depends in a subtle way on the choice of $\Omega$.

\begin{ex}\label{ex:specialcases}
Consider the conditional independence statement $\,1 \independent 3 \given 4\,$ 
for the DAG in Figure \ref{fig:4node}. The partial correlation is represented by the almost-principal minor
$$ \det(K_{12,23}) = a_{13} \cdot g \qquad
\hbox{where} \quad g \,=\,a_{23}a_{24}a_{34}+a_{24}^2+1.
$$
The component $\{g=0\}$ is smooth in $\R^4$. However,
it is disjoint from the cube $\Omega = [-1,1]^4$.
To see this, note that $-1 \leq a_{23}a_{24}a_{34}$ in $\Omega$.
With this,  $g=0$ would imply
$a_{24} = 0$ and hence $g=1$, a contradiction. Consequently, if $\Omega$
is the cube $[-1,1]^4$ then
the correlation hypersurface is simply $\{a_{13} =0\}$,
and the RLCT equals $(1,1)$ by Proposition \ref{prop:RLCTMonomial}.
The other special case with ${\rm RLCT} = (1,1)$ in
Table 2  comes from swapping the labels of nodes $1$ and~$2$.

Now, if we enlarge the parameter space $\Omega$ then the situation
changes. For instance, suppose $(a_{13},a_{23},a_{24},a_{34}) = (0,-2,1,1)$
is in the interior of $\Omega$. This is a singular point of
$\det(K_{12,23}) = a_{13} \cdot g$. The RLCT can be 
computed by applying Theorem \ref{thm:RLCTMonomial2}. It is now
 $(1,2)$ instead of
$(1,1)$. This example shows that the asymptotics of  ${\rm
  V}_{i,j|S}(\lambda)$ depends on $\Omega$. However, it is
  possible to choose $\Omega$ in such a way that further enlargement
  will not cause the asymptotics of $V_G(\lambda)$ to change. Such a choice could be used as a worst-case analysis for
  $V_G(\lambda)$, but to avoid complicating the
  paper, we will not explore this any further. \qed
\end{ex}

\begin{figure}[t!]
\centering
\includegraphics[scale=0.28]{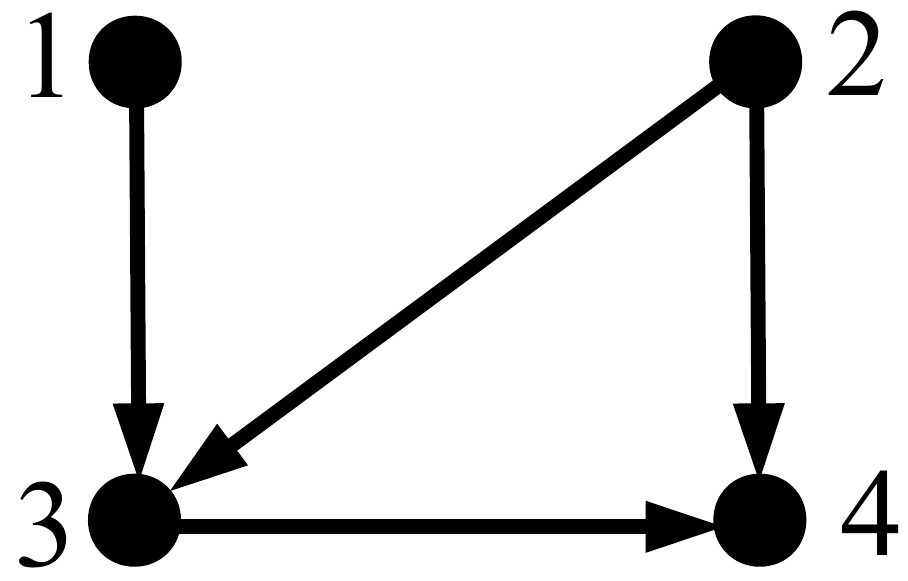}
\caption{4-node DAG.}
\label{fig:4node}
\end{figure}

\begin{rem}
We briefly return to the issue of faithfulness in the PC-algorithm.
Zhang and Spirtes \cite{ZhangSpirtes08} introduced a variant known as
 the \emph{conservative PC-algorithm}. As the name suggests, this algorithm is more conservative and may decide not to orient certain edges. The conservative PC-algorithm only requires \emph{adjacency-faithfulness} for correct inference, which is simply strong-faithfulness  restricted to the edges of~$G$:
 $$  \qquad
|\corr(i,j|S)| > \lambda  \qquad
\hbox{for all $(i,j) \in E$ and
$S\subset V \backslash \{i,j\}$ }.
$$
If $\{i,j\}$ is not adjacent to $R$ then the relevant minor equals
 $\,\det(K_{iR,jR} ) =  a_{ij} {\rm det}(K_{R,R})
 +f(\bar{a})$, where $f$ is a polynomial in $\bar{a}=\{a_{st}\mid (s,t)\neq(i,j)\}$,
  the correlation hypersurface is smooth, and $(\ell, m ) = (1,1)$.
If $\{i,j\}$ is adjacent to $R$ then the behavior can be more complicated, as seen in
Example~\ref{ex:specialcases}.
\end{rem}

\section{Asymptotics for trees}
\label{sec:trees}

In \cite{URBY} trees were treated as one class.  However, as  noted when discussing Figure \ref{fig:graphs_lambda}, there is a striking difference between the volume $V_G(\lambda)$ for chain-like graphs compared to stars. In this section we give an explanation for this difference based on real log canonical thresholds.
  
We use the notation $SOS(a)$ for  any polynomial that is a sum of squares of polynomials
 in the model parameters $(a_{ij})_{(i,j)\in E}$. 
 Suppose that $G$ is a tree on $V = \{1,2,\ldots,p\}$
 and let $m$ be the longest length of an undirected path in $G$.
 It was shown in  \cite[Corollary 4.3(a)]{URBY} that
 any non-zero almost-principal minor of the concentration 
 matrix $K$ has the form
 \begin{equation}
 \label{eq:sosmono}
  \det(K_{iR,jR}) \, =\, (1+SOS(a)) \cdot a_{i \rightarrow j} , 
  \end{equation}
 where $a_{i \rightarrow j}$ is the monomial
 of degree $\leq m$ formed by multiplying the
 parameters $a_{rs}$ along the unique path between $i$ and $j$.
 Specifically, for the  two trees in Figure \ref{fig:graphs} we have
 $$
  \det(K_{iR,jR}) \, =\,  \begin{cases}\,
  (1+SOS(a))\prod_{k=i}^{j-1}a_{k,k+1} & {\rm if} \,\,\,
   G = \textrm{Chain}_p,
  \\ \, (1+SOS(a)) \cdot a_{1,i}a_{1,j} & {\rm if} \,\,\, G = \textrm{Star}_p, \textrm{ and }i,j> 1 .\\
  \end{cases}
$$
In both cases, the term $SOS(a)$ disappears when $i$ and $j$
are leaves of the tree $G$; cf.~(\ref{eq:chaincor}),(\ref{eq:starcor}).

Since the  correlation hypersurfaces for trees
are essentially given by monomials, we can apply Proposition \ref{prop:RLCTMonomial}.
The minimal real log canonical threshold is $(1,m)$
where $m$ is the largest degree of any of the monomials in 
(\ref{eq:sosmono}).
Corollary \ref{prop_unfaithful} implies the following result:

\begin{thm} 
\label{thm:treeV}
Under the assumptions in Theorem \ref{lemma_unfaithful}, if $G$ is a tree then
the volume of $\lambda$-strong-unfaithful distributions satisfies
$$
V_G(\lambda) \,\approx\, C \lambda (- \ln \lambda)^{m-1}
$$
where $m$ is the length of the longest path in the tree $G$,
and $C$ is a suitable constant. 
\end{thm}

For the case of stars we have $m = 2$,
whereas for chain-like graphs we have $m  = p-1$.

\begin{cor}
\label{for:treecor}
Under the assumptions in Theorem \ref{thm:treeV}, the volume $V_G(\lambda)$ of strong-unfaithful distributions satisfies
\begin{equation}
\label{eq:treeV}
V_G(\lambda) \,\approx\,  
\begin{cases}
\,\,C_{\rm chain} \cdot \lambda (-\ln \lambda)^{p-2} & \textrm{if $\,G = \textrm{Chain}_p$,}\\
\,\, C_{\rm star} \cdot \lambda (-\ln \lambda) & \textrm{if $\, G = \textrm{Star}_p$.}
\end{cases}
\end{equation}
where $C_{\rm chain}$ and $C_{\rm star}$ are suitable positive constants.
\end{cor}

As a consequence, the volume $V_G(\lambda)$
  is asymptotically larger for chains compared to stars,
   and the difference increases with increasing number of nodes $p$. This furnishes an explanation for 
Figure \ref{fig:graphs_lambda}, at least for small 
   values of $\lambda$. In that figure we saw the curve for the chain
lying clearly above the curve for the star tree.
However, one subtle issue is the size of the constants
$C_{\rm chain}$ and $C_{\rm star}$. These need to be understood
in order to make accurate comparisons.

In Section \ref{sec:integrals}, we develop new theoretical results regarding the computation 
of the constant $C$ in (\ref{eq:firstorder}).
Theorem~\ref{thm_constants} gives an
  integral representation for $C$ when the partial correlation hypersurface is essentially defined by a monomial.
In Example \ref{ex:chainStarConstants} we shall then derive:

\begin{cor} 
\label{cor__const_trees}
The two constants in (\ref{eq:treeV}) are 
$$
C_{\rm chain} \, = \, \frac{1}{(p-2)!} 
\quad \hbox{and} \quad 
C_{\rm star} \, = \,\binom{p}{3}. 
$$
\end{cor}

This result surprised us at first. It establishes the counterintuitive
fact that, as $p$ grows, the constant for the
lower curve in  Figure~\ref{fig:graphs_lambda} 
 is exponentially larger  than that for the upper curve.
 Therefore, in order to fully explain the
 relative position of the two curves for
 a wider range of values of $\lambda > 0$, it does not
 suffice to just consider the first order
 asymptotics (\ref{eq:firstorder}).
 Instead, we need to consider some of the higher order terms in the
series expansion (\ref{eq:asymp}).

As we shall see in Section \ref{sec:integrals}, it is difficult 
to determine the constants $C_{\ell,m}$ in (\ref{eq:asymp})  analytically.
In the remainder of this section, we propose a procedure
based on simulation and linear regression for estimating the constants 
$C_{\ell,m}$ in the asymptotic explanations of the volumes $V_G(\lambda)$
or $V_{i,j|S}(\lambda)$. For simplicity
we focus on the latter case and we take~$f = {\rm det}(K_{iR,jR})$.

Suppose that $G$ is a DAG for which 
the real log canonical thresholds $(\ell,m)$ in
Theorem \ref{lemma_unfaithful} and
Corollary \ref{prop_unfaithful} are known.
This is the case for all trees by Theorem \ref{thm:treeV}.
Our procedure goes as follows. We first sample $n$ points uniformly from $\Omega$ and compute the proportion of points $\omega$ that lie in
$\, {\rm Tube}_{i,j|S}(\lambda)\,$
 for different values of $\lambda$. We then fit a linear model to
$$\frac{V_{i,j|S}(\lambda)}{\lambda^{\ell}} \,\,\approx \,\,C_{m-1}(-\ln\lambda)^{m-1} + C_{m-2}(-\ln\lambda)^{m-2}+\cdots + C_0,$$
where $(\ell, m)$ is the known real log canonical threshold. 

In the following, we illustrate this procedure for chains and stars. We analyze two specific examples of partial correlation volumes, namely the ones corresponding to the longest paths in each graph, that is $V_{1,p|\emptyset}(\lambda)$ for ${\rm Chain}_p$ and
$V_{2,3|\emptyset}(\lambda)$ for ${\rm Star}_p$.
  For chain-like graphs,
\begin{equation}
\label{eq:chaincor}
\corr(1,p)\,=\,
\frac{(-1)^{p}\prod_{i=1}^{p-1}a_{i,i+1}}{\sqrt{1+a_{p-1,p}^2\left(1+a_{p-2,p-1}^2\left(\cdots(1+a_{12}^2)\right)\right)}},
\end{equation}
whereas for star graphs,
\begin{equation}
\label{eq:starcor}
\corr(2,3)\,\,=\,\,-\frac{\,a_{12}\,a_{13}}{\sqrt{(1+a_{12}^2)(1+a_{13}^2)}}.
\end{equation}

\begin{figure}[t!]
\centering
\subfigure[$p=6$]{\includegraphics[scale=0.42]{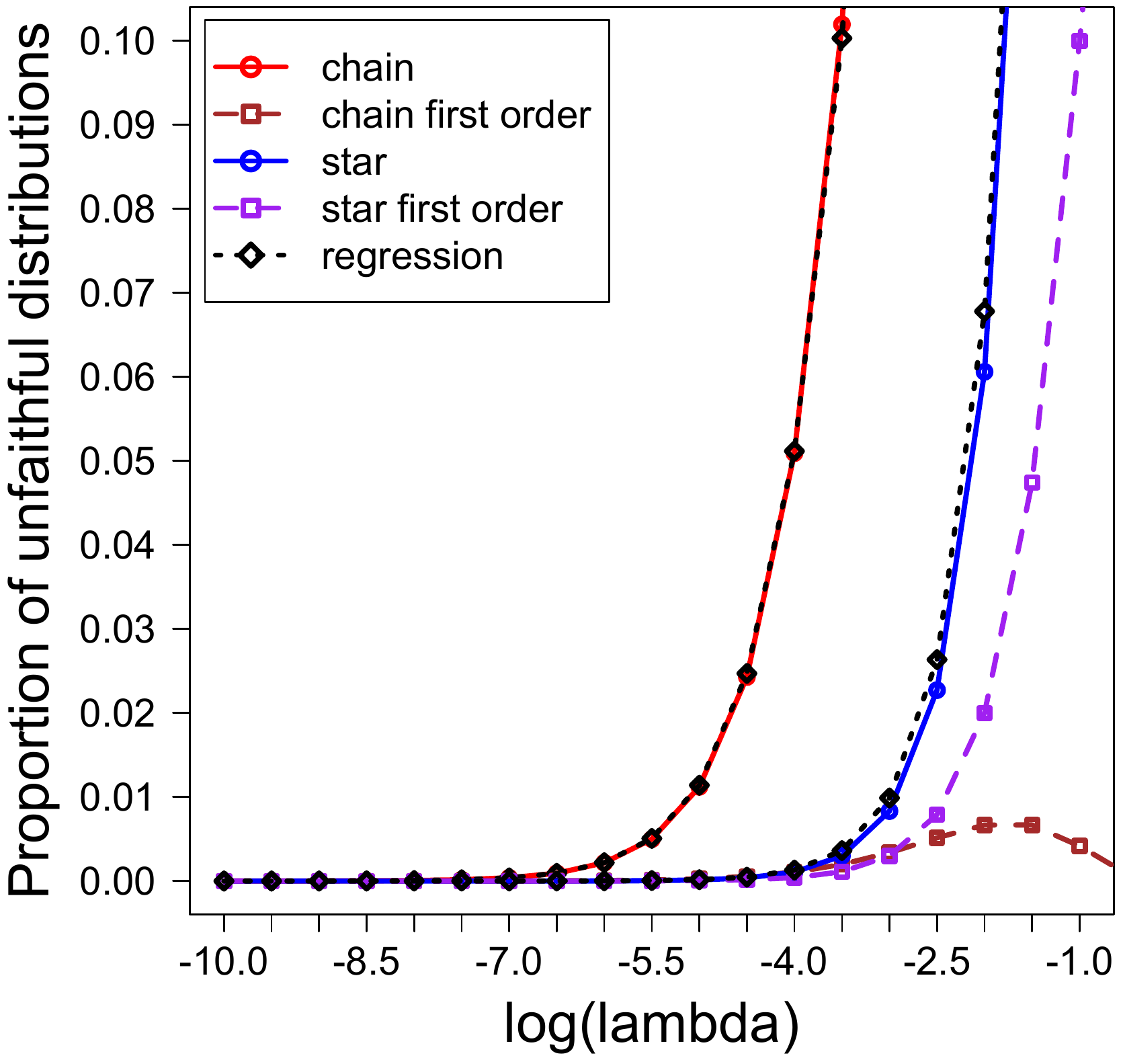}\label{Rplot_chain_star_reg_6}}\qquad\quad
\subfigure[$p=10$]{\includegraphics[scale=0.42]{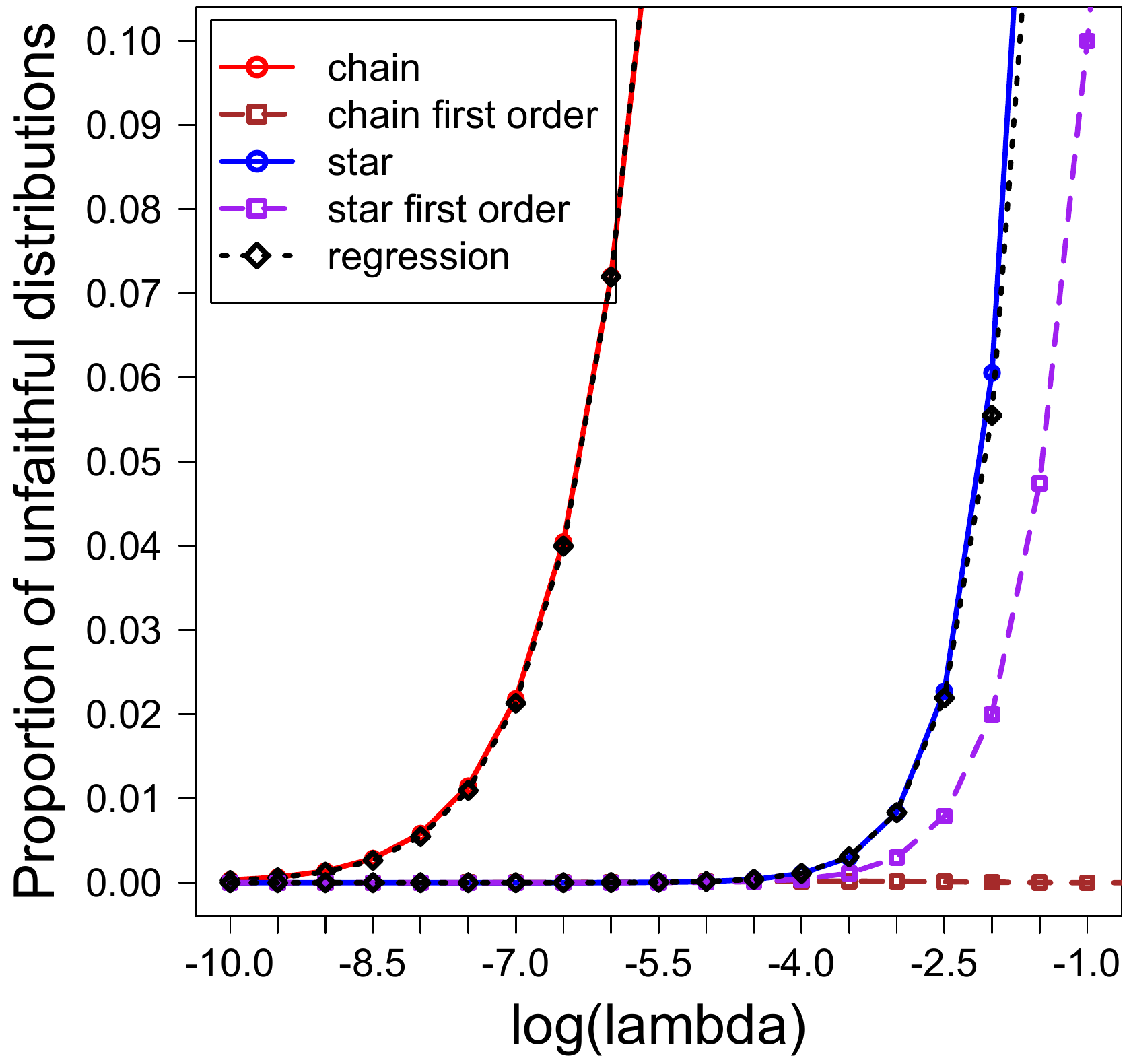}\label{Rplot_chain_star_reg_10}}
\caption{Regression-based asymptotics for chains and stars.}
\label{Rplot_chains_stars}
\end{figure}

We first approximate $V_{1,p|\emptyset}(\lambda)$ for chain-like graphs and $V_{2,3|\emptyset}(\lambda)$ for star graphs by simulation for various values of $\lambda$. This means
 that we sample $n$ points uniformly in the $(p-1)$-dimensional parameter space $\Omega$,
 and we count how many of them are $\leq \lambda$. The results for $p=6$ and $p=10$ are shown in Figure \ref{Rplot_chains_stars}. These are based on a sample size of
$n =  1,000,000$.
We then fit a linear model
$$\frac{V_{1,p|\emptyset}(\lambda)}{\lambda} \,\,\approx \,\,C_{p-2}(-\ln\lambda)^{p-2} + C_{p-3}(-\ln\lambda)^{p-3}+\cdots+C_0$$
for chain-like graphs. The curve resulting from the regression estimates is shown in black in Figure 
\ref{Rplot_chains_stars}. 
The curve resulting from the first-order approximation with the constants computed using Corollary \ref{cor__const_trees} is shown in grey in Figure \ref{Rplot_chains_stars}. We note that especially for chain-like graphs, where the true constant in Corollary \ref{cor__const_trees} is small,  the first order approximation is very bad.
  
   The approximation by regression on the other hand is a fast way to get pretty accurate estimates of all constants. 
The same was done with star graphs, but with the linear model
$$\frac{V_{2,3|\emptyset}(\lambda)}{\lambda} \,\,\approx \,\, C_{1}(-\ln\lambda)+C_{0}. 
$$
Figure \ref{Rplot_chains_stars} shows that the
first-order approximation is more accurate for stars than for chains.

%

\section{Volume inequalities for bias reduction in causal inference}
\label{sec:bias}

We now discuss the problem of quantifying bias in causal models. 
Our point of departure is Greenland's paper \cite{Greenland}, where the problem of quantifying bias has been discussed for binary variables. In contrast to the previous sections, in the situation 
 discussed here,
 a large tube volume is in fact desired since it corresponds to small~bias.
In this section we use the notation $K_{i,j|S}$ for the almost-principal
minor $K_{iR,jR}$ of the concentration matrix.

We are interested in estimating the direct effect of an exposure $E$ on a disease outcome $D$ 
(i.e.~the coefficient on the edge $E\to D$) from the partial correlation $\corr(E, D\,|\, S)$, where $S$ is a subset of the measurable variables. This partial correlation is a weighted sum over all {\em open paths} (i.e~paths which d-connect $E$ to $D$) given $S$ (the direct path $a_{ED}$ being just one of them). For estimating the direct effect $a_{ED}$ from $\corr(E, D\,|\, S)$, all open paths other than the direct path are thus considered as bias. We shall analyze two forms of bias which are of particular interest in practice, namely confounding bias and collider-stratification bias.  We start by defining collider-stratification~bias.

Suppose we are given a DAG $G$ with  $D,E\in V$ and there is another node $C$ such that 
$$E\rightarrow V_1 \rightarrow \dots \rightarrow V_s \rightarrow C\leftarrow W_1 \leftarrow \dots \leftarrow W_t \leftarrow D. $$ 
This says that
$C$ is a collider on a path from $D$ to $E$. Stratifying (i.e.~conditioning) on $C$ opens a path between $E$ and $D$ leading to bias when estimating $a_{ED}$. The partial correlation corresponding to the opened path between $E$ and $D$ is known as \emph{collider-stratification bias}. Collider-stratification bias arises for example in the context of discrete variables, where instead of obtaining a random sample from the full population,  a random sample is obtained from the subpopulation of individuals with a particular level of C.

\begin{figure}[t!]
\centering
\subfigure[$\textrm{Tripart}_{5,1}$]{\includegraphics[scale=0.28]{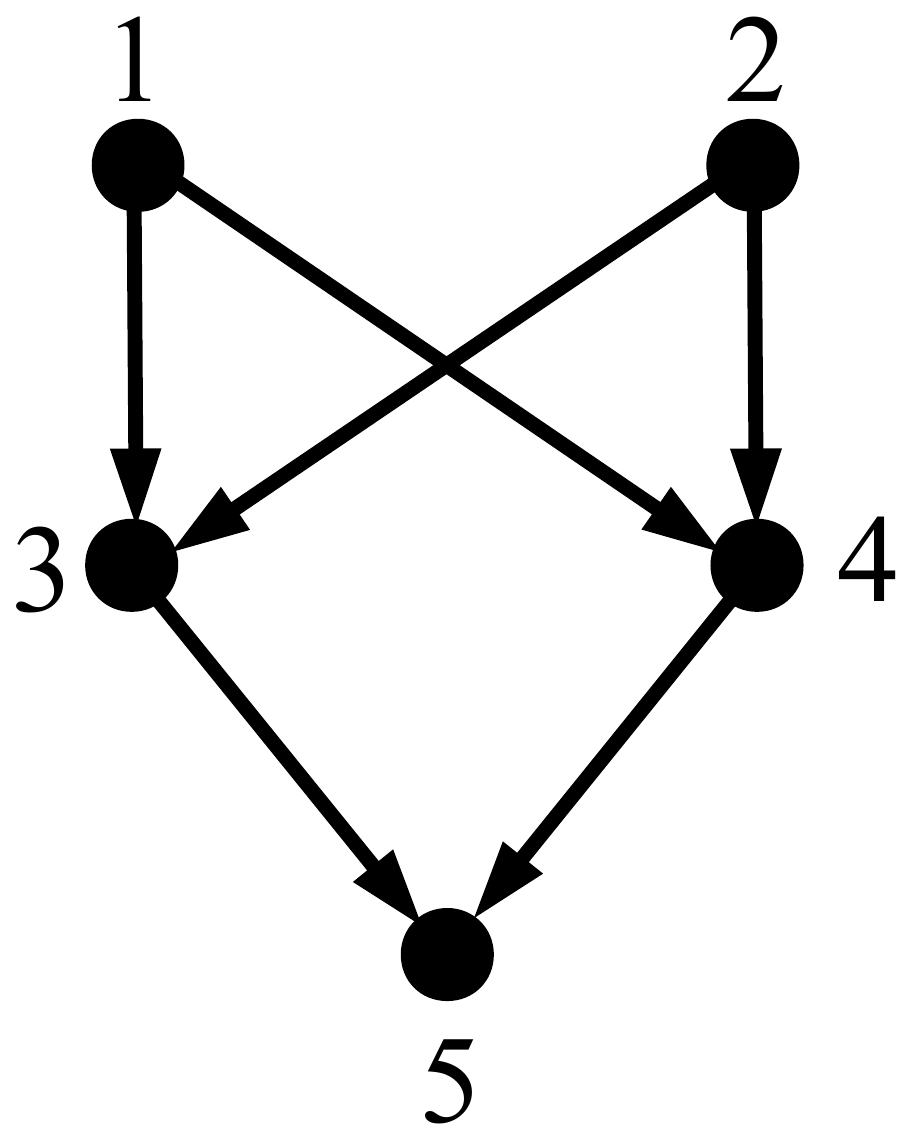}\label{fig:tripart_51}}\qquad\quad
\subfigure[$\textrm{Tripart}_{6,2}$]{\includegraphics[scale=0.28]{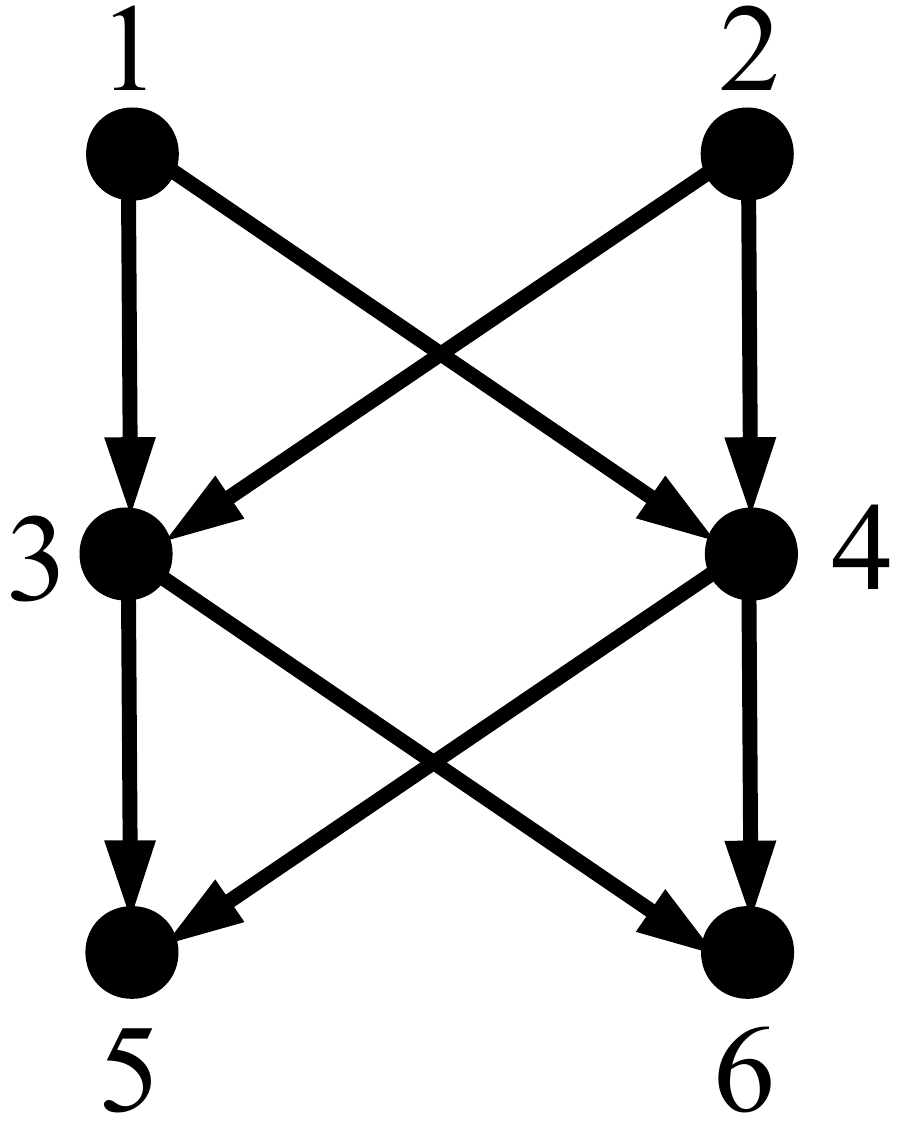}\label{fig:tripart_62}}\qquad\quad
\subfigure[$\textrm{Tripart}_{p,p-3}$]{\includegraphics[scale=0.28]{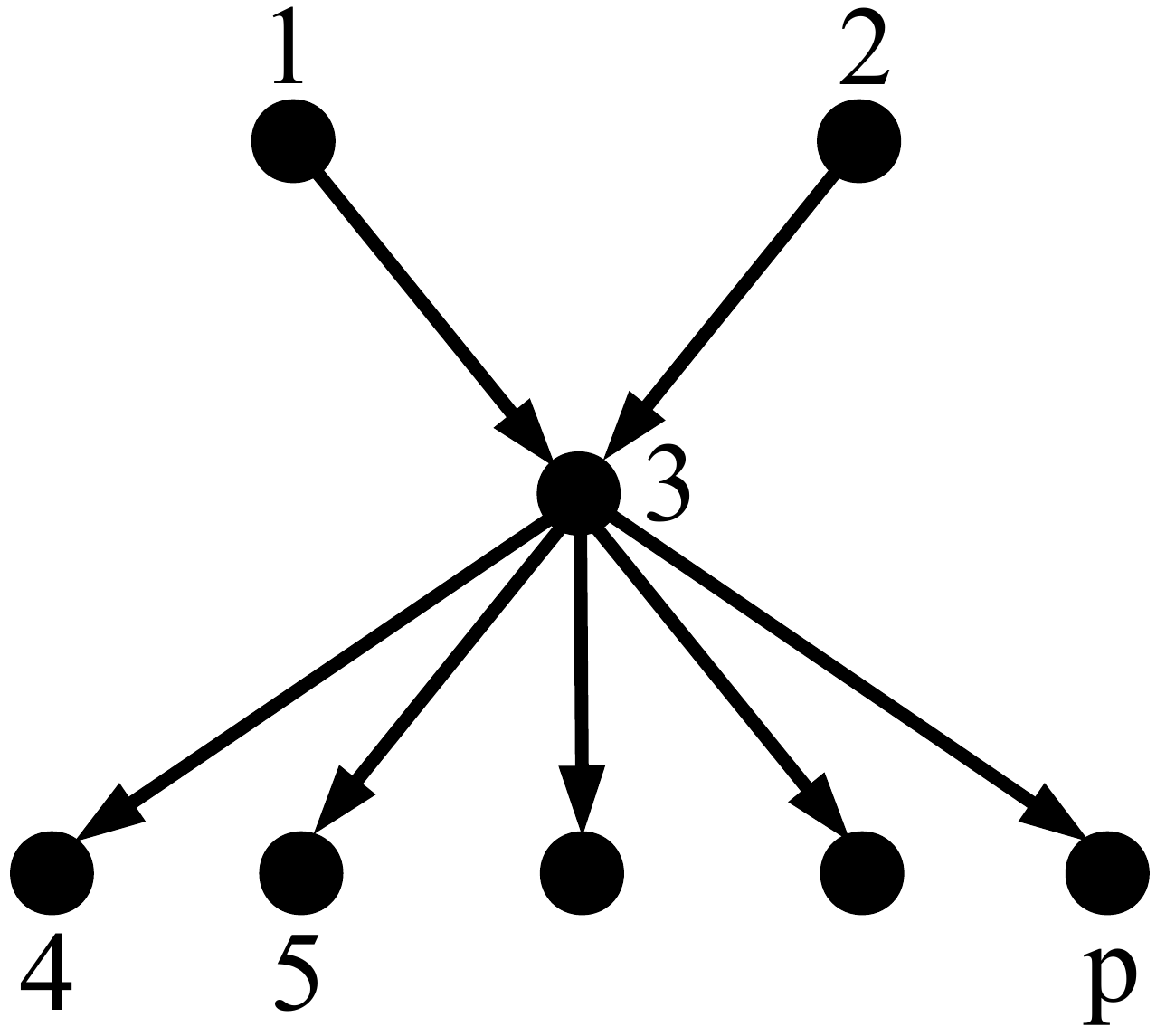}\label{fig:tripart_p}}\qquad\quad
\subfigure[Almost-chain graph]{\includegraphics[scale=0.28]{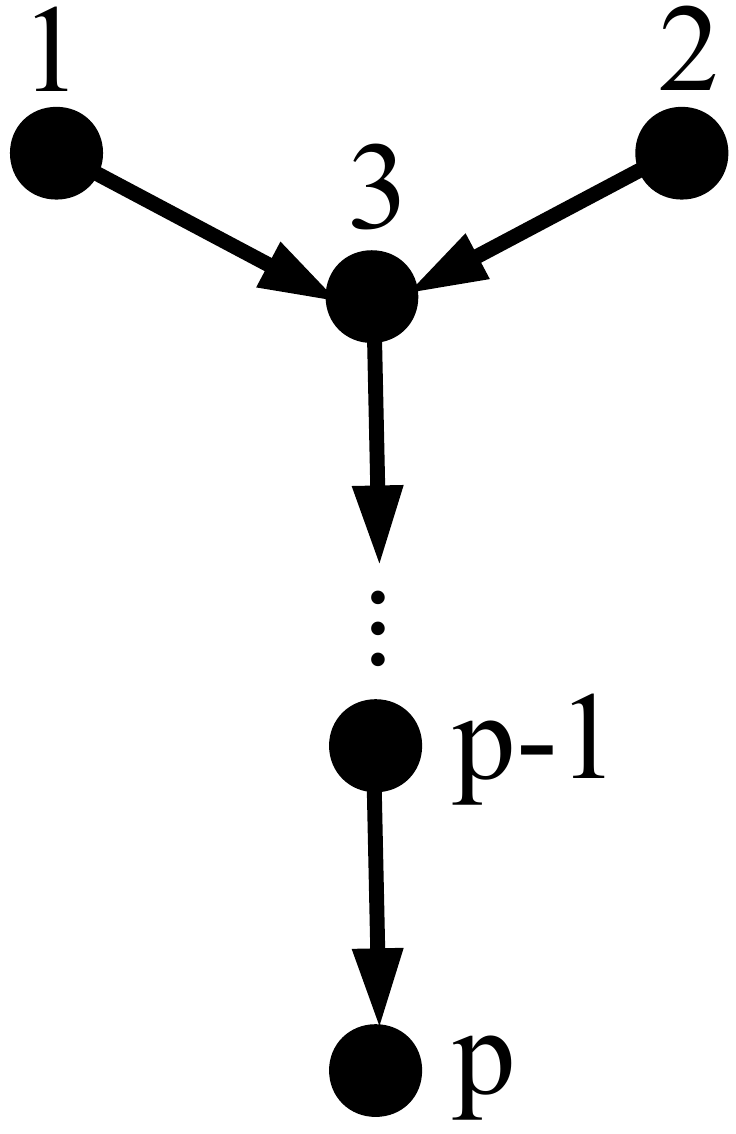}\label{fig:almost_chain}}
\caption{Various tripartite and almost tripartite graphs.}
\label{fig:all_tripartite}
\end{figure}

\begin{ex}
\label{ex:tripartitecollider}
We illustrate collider-stratification bias for the tripartite graph $G = {\rm Tripart}_{5,1}$ shown in Figure \ref{fig:tripart_51}. Let node
$1$ represent the exposure $E$ and node $2$ the disease outcome $D$. In this example, node $5$ is a collider
$C$ for multiple paths between $E$ and $D$. When stratifying  on
$C= 5$,  node $E=1$  is d-connected to node $D=2$  via the following paths:
$$1\rightarrow 3 \rightarrow 5 \leftarrow 4\leftarrow 2, \quad 1\rightarrow 4 \rightarrow 5 \leftarrow 3\leftarrow 2, \quad1\rightarrow 3 \rightarrow 5 \leftarrow 3\leftarrow 2, \quad  1\rightarrow 4 \rightarrow 5 \leftarrow 4\leftarrow 2.
$$
The bias introduced for estimating the direct effect of $E$ on $D$ when conditioning on $C$ is
\begin{equation}
\label{eq:bias1}
\corr(1, 2 \,|\, 5)  \,\,=\,\, \frac{a_{13}a_{35}a_{45}a_{24} + a_{14}a_{45}a_{35}a_{23} + a_{13}a_{35}^2a_{23} + a_{14}a_{45}^2a_{24}}{ \sqrt{\det(K_{134,134})\det(K_{234,234})}}.
\end{equation}
The numerator ${\rm det}(K_{134,234})$ is the weighted sum of all open paths between $E$ and $D$. 
Similarly, nodes $3$ and $4$ are colliders for multiple paths. The bias when conditioning on these~is
\begin{equation} 
\label{eq:bias2}
\corr(1, 2\,|\, 34) \,\,= \,\,\corr(1,2\,|\,345) \,\, = \,\,
\frac{a_{13}a_{23}+a_{14}a_{24}}
{\sqrt{(a_{13}^2+a_{14}^2+1)(a_{23}^2+a_{24}^2+1)}}.
\end{equation}
Problem \ref{prob:tripart} is about
comparing the tube volume for (\ref{eq:bias1})
with the tube volume for  (\ref{eq:bias2}).
\qed
\end{ex}

A question of practical interest in causal inference
is to understand the situations in which stratifying on a collider leads to a particularly large bias. It is widely believed that collider-stratification bias tends to attenuate when it arises from more extended paths  (see \cite{Chaudhuri, Greenland}). 
What follows is our interpretation of this statement as a precise mathematical conjecture.

\begin{prob}
\label{prob:tripart}
Let $D, E\in V$ and $\,
\mathcal{C} =\{C\in V \mid \exists \textrm{ path } P \textrm{ from $E$
  to $D$ with $C$ as a collider} \}$.
We introduce a partial order on 
the {\em collider set} $\,\mathcal{C}$ by setting $C\leq C'$ if  all paths on which $C$ is a collider also go through $C'$. 
Given subsets $S,S' \subset \mathcal{C}$ we set $S \leq S'$ if 
for all $C \!\in\! S$ there exists $C'\!\in\! S'$ such that $C \leq C'$.
If this holds, then the bias introduced when conditioning on $S$
   should be smaller than when conditioning on $S'$.
To make this precise, we conjecture:
\begin{equation}
\label{eq:conjecture1}
V_{D,E \mid S}(\lambda) \,\geq\, V_{D,E \mid S'}(\lambda) \qquad \textrm{for all }
S \leq S' \,\, \textrm{and all} \,\, \lambda\in[0,1].
\end{equation}
\end{prob}
\medskip

\begin{figure}[t!]
\centering
\subfigure[Complete tripartite graphs]{\includegraphics[scale=0.38]{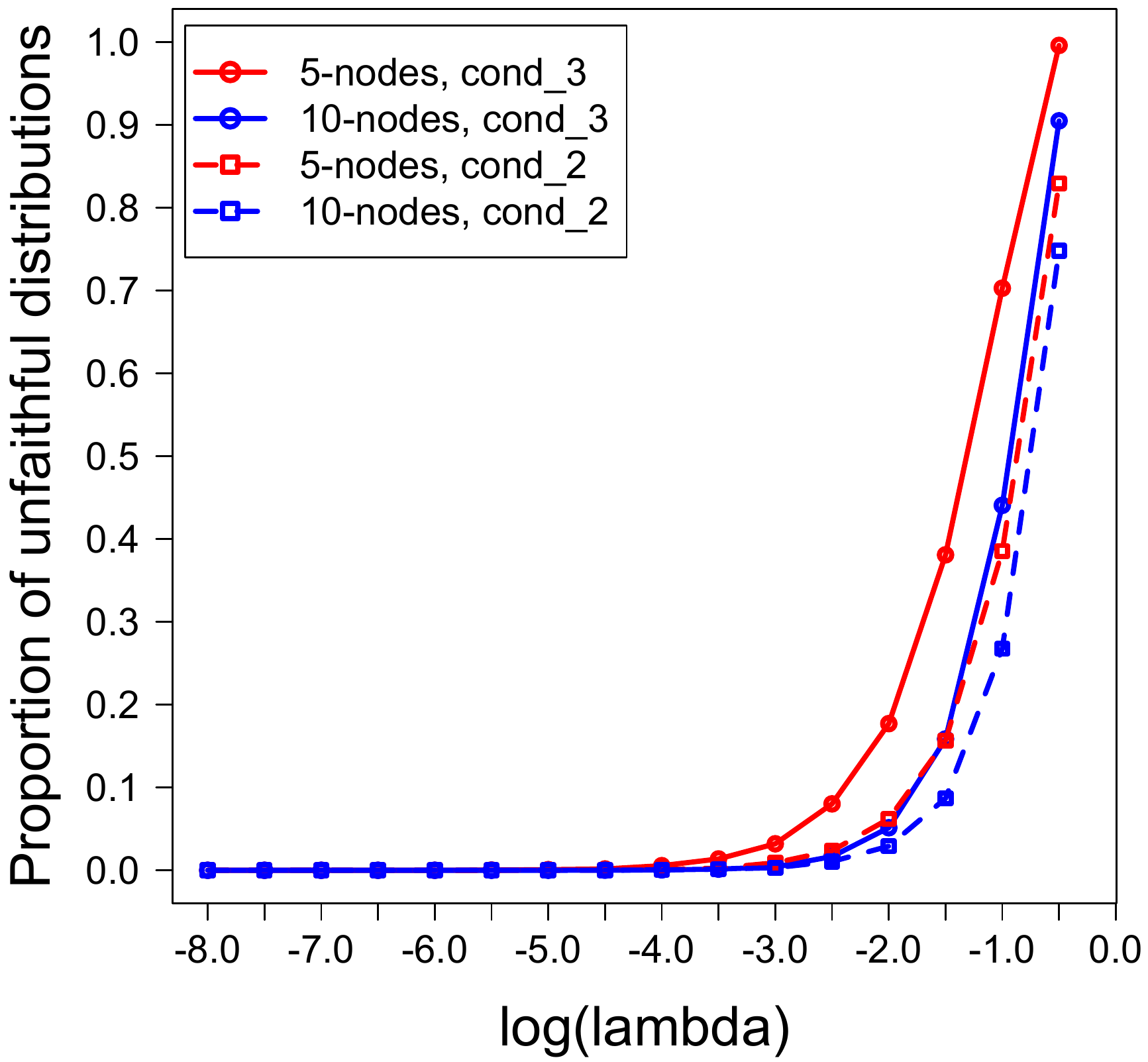}\label{Rplot_tripartite}}\qquad\quad
\subfigure[Bow-ties]{\includegraphics[scale=0.38]{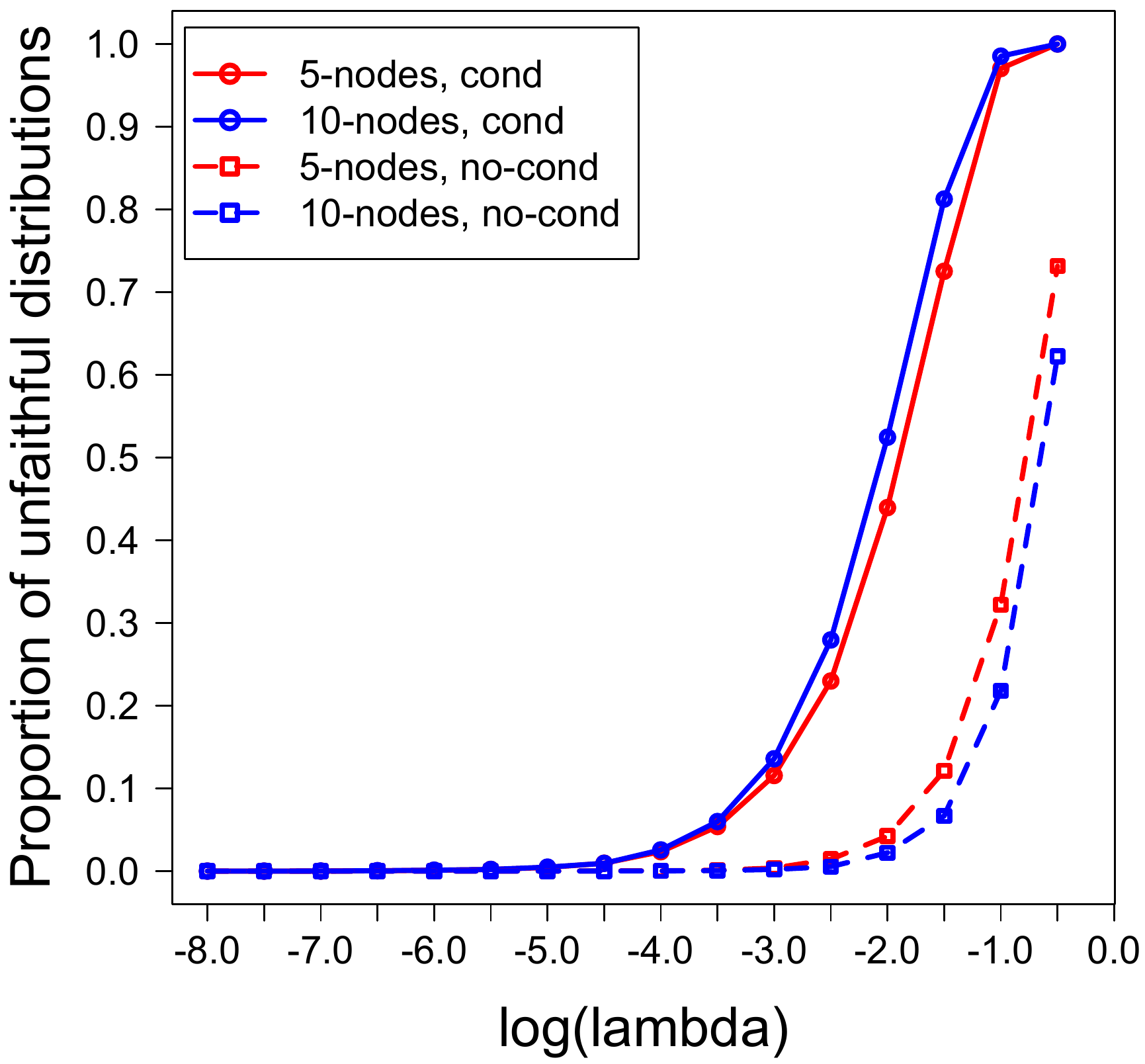}\label{Rplot_bowties}}
\caption{Effect of collider-bias on complete tripartite graphs and bow-ties.}
\label{fig:Mbias}
\end{figure}

We now study this conjecture for the tripartite graphs $\textrm{Tripart}_{p,p'}$.
Here, it says that the collider-stratification bias introduced when conditioning on the third level 
  $\{p-p'+1,\ldots, p\}$ is in general smaller than when conditioning on the second level of 
  nodes $\{3,\ldots, p-p'\}$, i.e.
\begin{equation}
\label{conj_tripart}
V_{1,2 \mid p-p'+1,\ldots, p}(\lambda) \,\,\geq \,\,
V_{1,2 \mid 3,\ldots, p-p'}(\lambda). 
\end{equation}
This inequality is confirmed by the simulations shown in Figure \ref{Rplot_tripartite}.
Here $p=5$, $p'=2$ is shown in red and $p=10$, $p'=2$ is shown in blue. The solid lines correspond to the volume $V_{1,2 \mid p-p'+1,\ldots, p}(\lambda)$, whereas the dashed lines correspond to the volume $V_{1,2 \mid 3,\ldots, p-p'}(\lambda)$.

Going beyond simulations, we now present
  an algebraic proof of
  our conjecture when $\lambda$ is small, for the tripartite graphs in
  Figure \ref{fig:tripart_p} where the second level has only one node.

\begin{ex} 
\label{ex:tripart_1}
For $G = \textrm{Tripart}_{p,p-3}$ the left hand side of
(\ref{conj_tripart}) is given by
$$
\det(K_{1,2\mid 4,5,\ldots p}) \,\, = \,\, a_{13}a_{23}(\sum_{k=4}^p a_{3k}^2).
$$
Depending on the values of $p$, the corresponding real log canonical threshold is given by
\begin{equation}
\label{RLCT_tripart_1}
\RLCT(1,2|4,\ldots,p)\,\,= \,\,\begin{cases} 
\,(\frac{1}{2}, 1)& \textrm{if $\,p=4$,}\\
\,(1, 3)& \textrm{if $\,p=5$,}\\
\,(1, 2)& \textrm{if $\,p\geq 6$.}
\end{cases}
\end{equation}
 For $p=4$ this was Example \ref{ex:sec4graph2}.
 To prove (\ref{RLCT_tripart_1}) for $p\geq 5$, we need two
  ingredients. Firstly, if the polynomial is a product of factors
  with disjoint variables, then the RLCT is the minimum of the RLCT of
  the factors, taken with multiplicity (e.g. if the RLCTs are
  $(\ell,m_1)$ and $(\ell,m_2)$, then the combined RLCT is $(\ell,
  m_1+m_2)$, just like in the case of a
  monomial). Secondly, the RLCT of a sum of squares of $d$ unknowns is equal to
     $(d/2,1)$. We saw this in  Example \ref{ex:d-ball}.
  
For the right hand side of (\ref{conj_tripart}), we condition
  on node $3$. Now, the defining polynomial~is 
$$
\det(K_{1,2\mid 3})\,\,= \,\, a_{13}a_{23}.
$$
By Proposition \ref{prop:RLCTMonomial},
this has RLCT $(1,2)$, which is larger or
equal to all values of $(\ell, m)$ in (\ref{RLCT_tripart_1}). To
compare the behavior of $V_{1,2 \mid 3}(\lambda)$ and $V_{1,2 \mid
  4,\ldots, p}(\lambda)$ for small $\lambda$, we will need to derive
the constant $C$ in (\ref{eq:firstorder}). In Example
\ref{ex:bowtieConstant}, we will show that if $p \geq 6$ and the parameter space
is
\begin{equation}\label{eq:newParamSpace}
\Omega \,\,=\,\, \{a \in \R^{p-1} : |a_{12}| \leq 1, |a_{23}| \leq 1, 
a_{34}^2+\cdots + a_{3p}^2 \leq 1\},
\end{equation}
then the asymptotic constants are given by $C_{1,2|3} = 1$ and
$C_{1,2|4,\ldots,p} = 2+ 2/(p-5)$. 
We conclude that $V_{1,2 \mid 3}(\lambda) \leq V_{1,2 \mid 4,\ldots, p}(\lambda)$ 
for small values of $\lambda $, as conjectured in Problem \ref{prob:tripart}.
\qed
\end{ex}

\begin{ex} A slight twist to Example \ref{ex:tripart_1} is the almost-chain graph shown in Figure \ref{fig:almost_chain}, with edges $E=\{(1,3), (2,3), (3,4), \ldots (p-1,p)\}$. For such graphs, Problem \ref{prob:tripart} asks whether
$$V_{1,2 \mid s}(\lambda) \leq V_{1,2 \mid t}(\lambda) \qquad \textrm{if }s\leq t.$$
This holds for small $\lambda$ because $\det(K_{1,2\mid s})=a_{13}a_{23}\prod_{k=3}^{s-1}a_{k,k+1}^2$.
By Proposition \ref{prop:RLCTMonomial}, 
$$
\RLCT(1,2|s)\,\,= \,\,\begin{cases} 
\,(1, 2)& \textrm{if $\,s=3$,}\\
\,(\frac{1}{2}, s-3)& \textrm{if $\,s\geq 4$,}
\end{cases}
$$
so $\RLCT(1,2|s) \geq \RLCT(1,2|t)$ for $s \leq t$. \qed
\end{ex}

In Example \ref{ex:tripart_1} we resolved Problem \ref{prob:tripart} for
tripartite graphs whose middle level consists of one node.
We next consider the case 
${\rm Tripart}_{p,1}$
where the third level has one node.

\begin{ex}
\label{ex:tripart_1third}
The graph $\textrm{Tripart}_{5,1}$ shown in Figure
\ref{fig:tripart_51} was discussed in Example~\ref{ex:tripartitecollider}. We focus on
   the numerators in (\ref{eq:bias1}) and in (\ref{eq:bias2}). 
  The polynomial (\ref{eq:bias1}) will be studied in
  Example \ref{ex:tripartite_blowup} where we prove that
  $\textrm{RLCT}(1,2 \,|\, 5) = (1,3)$. Using the same
  method for $\textrm{Tripart}_{p,1}$ gives 
$$
\RLCT(1,2|p) \,\,=\,\,\begin{cases} 
\,(\frac{1}{2}, 1)& \textrm{if $\,p=4$,}\\
\,(1, 3)& \textrm{if $\,p=5$,}\\
\,(1, 2)& \textrm{if $\,p\geq 6$,} 
\end{cases} \qquad 
\RLCT(1,2|3,\ldots,p-1) \,\,=\,\,\begin{cases} 
\,(1, 2)& \textrm{if $\,p=4$,}\\
\,(1, 1)& \textrm{if $\,p\geq 5$.}
\end{cases} 
$$
Thus, we conclude that $V_{1,2 \mid p}(\lambda) \geq V_{1,2
  \mid 3,\ldots, p-1}(\lambda)$ for small values of $\lambda> 0$. \qed
\end{ex}

\begin{ex} \label{ex:tripart_62}
For the graph $\textrm{Tripart}_{6,2}$ in Figure \ref{fig:tripart_62}
we check if $V_{1,2 \mid 56}(\lambda)\geq V_{1,2
  \mid 34}(\lambda)$ for small $\lambda$.
As before, $\textrm{RLCT}(1,2\,|\, 3,4) = (1,1)$, but
a hard  computation using the tools of Section \ref{sec:resolution} 
 reveals that now $\textrm{RLCT}(1,2\,|\, 5,6) = (1,1)$.
Thus, knowledge of the RLCT is not sufficient to
establish (\ref{eq:conjecture1}). What is needed is a finer analysis
along the lines of Section~\ref{sec:integrals}.
\qed
\end{ex}

\smallskip

The second form of bias studied by Greenland \cite{Greenland} is confounder bias. 
In the context of a directed graphical model $G$,
a \emph{confounder} for the effect of $E$ on $D$ is a node $C$ such that
$$E\leftarrow V_1 \leftarrow \dots \leftarrow V_s \leftarrow C\rightarrow W_1 \rightarrow \dots \rightarrow W_t \rightarrow D.$$ 
The partial correlation introduced by the path from $E$ to $D$ passing through $C$ is 
referred to as \emph{confounder bias}. In such situations, stratifying on $C$ blocks 
the path between $E$ and $D$ (i.e. $C$ d-separates $E$ from $D$) and therefore corresponds to bias removal. 

In certain graphs, such as the bow-tie example in \cite{Greenland}, there are variables where stratifying removes confounder bias but at the same time introduces collider-stratification bias. For instance, 
 consider the graph $G = {\rm Bow}_{5}$, where node $4$ corresponds to exposure $E$ and node 
 $5$ corresponds to disease outcome $D$. Then
conditioning on node $3$  blocks the paths 
\begin{equation}
\label{paths_1}
4\leftarrow 3 \rightarrow 5, \quad 4\leftarrow 1 \rightarrow 3 \rightarrow 5, \quad 4\leftarrow 3 \leftarrow 2 \rightarrow 5,\end{equation}
and therefore reduces confounder bias, but opens the path
\begin{equation}
\label{paths2}
4\leftarrow 1 \rightarrow 3 \leftarrow 2 \rightarrow 5
\end{equation} 
and therefore introduces collider-stratification bias. This trade-off is of particular interest in situations where one cannot condition on $1$ and $2$, for example because these variables were unmeasured.  It is believed that in such examples the bias removed by conditioning on the confounders is larger than the collider-stratification bias introduced and one should therefore stratify. We translate this statement into the following mathematical problem:

\begin{prob}
\label{prob:bowtie}
Let $D, E\in V$ and we denote by $\mathcal{D}$ the confounder-collider subset, i.e.
$$
\mathcal{D} \,\,=\,\,\mathcal{C} \,\cap \,
\{C\in V \mid \exists \textrm{ path } \pi \textrm{ from $E$ to $D$ having $C$ as a confounder} \}.
$$
We conjecture the following inequality for the relevant tube volumes:
\begin{equation*}
V_{D,E \mid S}(\lambda) \,\geq\, V_{D,E \mid \emptyset}(\lambda), \qquad \textrm{for all }
S\subset\mathcal{D} \,\,\textrm{and all} \,\,\lambda\in[0,1].
\end{equation*}
\end{prob}
\medskip

This conjectural inequality is interesting for the bow-tie graphs $\textrm{Bow}_{p}$.
It   means that conditioning on the nodes in the second level reduces bias since the bias removed by conditioning on the confounders is larger than the collider-stratification bias introduced by conditioning:
\begin{equation}
\label{conj:bow}
V_{p-1,p \mid 3,\ldots, p-2}(\lambda)\,\, \geq \,\,
V_{p-1,p \mid \emptyset}(\lambda). 
\end{equation}
This is confirmed by our simulations in Figure \ref{Rplot_bowties}, for $p=5$ in red and $p=10$ in blue. The solid line corresponds to the volume $V_{p-1,p \mid 3,\ldots, p-2}(\lambda)$ and the dashed line corresponds to
$V_{p,p-1 \mid \emptyset}(\lambda)$. In the following example we prove 
the inequality (\ref{conj:bow}) for $p=5$
and small~$\lambda > 0$.

\begin{ex}
Let $G = \textrm{Bow}_{5}$ as in Figure \ref{fig:bowtie}.
The  left hand side (\ref{conj:bow}) is represented by
$$
\det(K_{4,5\mid 3})\,\, = \,\, a_{13}a_{14}a_{23}a_{25}.
$$
This monomial is the path in (\ref{paths2}). The corresponding real log canonical threshold is $(1,4)$. 
The polynomial representing the right hand side (\ref{conj:bow}) is a weighted sum of the paths in (\ref{paths_1}):
$$
\det(K_{4,5\mid \emptyset}) \,\,=\,\, a_{34}a_{35}(1+a_{13}^2+a_{23}^2)+a_{23}a_{25}a_{34}+a_{13}a_{14}a_{35}.
$$
We derive its real log canonical threshold using the blowups described in Section \ref{sec:resolution}. We find that it is $(1,1)$. Since $(1,4)<(1,1)$, we 
conclude $V_{4,5 \mid 3}(\lambda) \geq V_{4,5 \mid \emptyset}(\lambda)$.
\qed
\end{ex}

\section{Normal crossing and blowing up}
\label{sec:resolution}

In this section we develop more refined techniques
for computing real log canonical thresholds.
The following theorem combines the monomial case of
Proposition \ref{prop:RLCTMonomial} with the
smooth case of Proposition \ref{thm:smoothFunctionRLCT}.
As promised in Section \ref{sec:real_log_threshold},
this furnishes the proofs for these two propositions.

\begin{thm} \label{thm:RLCTMonomial2}
Suppose $\varphi(\omega) = \omega_1^{\tau_1} \cdots
\omega_d^{\tau_d}$ and  $f(\omega) = \omega_1^{\kappa_1} \cdots
\omega_r^{\kappa_r} g(\omega)$ 
where $\tau_1, \ldots, \tau_d$ are
nonnegative integers, $\kappa_1, \ldots, \kappa_r$ are
positive integers, and the hypersurface $g(\omega)=0$ is either
  empty or smooth and normal crossing (see definition below) with $\omega_1, \ldots,
  \omega_r$.  We write $\omega_0 = g$, $\kappa_0 = 1$
    and $\tau_0 = 0$, and we define
$$\ell  = \min_{i \in \mathcal{I}} \frac{\tau_i+1}{\kappa_i},\quad \mathcal{J} =
\left\{
\operatorname*{argmin}_{i \in \mathcal{I}} \frac{\tau_i+1}{\kappa_i} \right\},
\quad m = |\mathcal{J}| , $$
where $\mathcal{I}$ is the set of all indices $0\leq i \leq r$ such
that $\omega_i$ has a zero in $\Omega$. 
 Then we have 
 $$ \RLCT_\Omega(f;\varphi) \,\, = \,\, (\ell,m), $$
 provided the equations
$\omega_i = 0$ for $i \in \mathcal{J}$
have a solution in the interior of $\Omega$.
\end{thm}

The normal crossing hypothesis in Theorem \ref{thm:RLCTMonomial2}
means that the system
$$
f \,= \,\omega_1\frac{\partial f}{\partial \omega_1}\,= \,\cdots\, = \,
\omega_r\frac{\partial f}{\partial \omega_r} \, =
 \,\frac{\partial f}{\partial \omega_{r+1}} \,= \,\cdots \,= \,\frac{\partial f}{\partial \omega_d} \,= \,0
$$
does not have any solutions in $\Omega$. See \cite{kollar} to learn more
about normal crossing singularities. 

We begin with  a technical lemma establishing that the RLCT
 can be computed locally.

\begin{lem} \label{lemma:RLCTLocalization}
For every $x\in \Omega$, there exists a
neighborhood $\Omega_x \subset \Omega$ of $x$ such that 
$$
\RLCT_{\Omega_x}(f;\varphi) \,=\, \RLCT_U(f;\varphi)
$$
for all neighborhoods $U \subset \Omega_x$ of $x$. Moreover, 
$$
\RLCT_\Omega(f;\varphi) \,=\, \min_{x} \RLCT_{\Omega_x}(f;\varphi)
$$
where we take the minimum over all $x$ in the real analytic hypersurface
$\{\omega \in \Omega : f(\omega) = 0\}$.
\end{lem}

\begin{proof}
This comes from \cite[Lemma 3.8, Proposition 3.9]{LinThesis}.
\end{proof}

\begin{proof}[Proof of Theorem \ref{thm:RLCTMonomial2}]  Lemma~\ref{lemma:RLCTLocalization} states that
  $\RLCT_\Omega(f;\varphi)$ is the minimum of $\RLCT_{\Omega_x}(f;\varphi)$ as $x$
  varies over $\Omega$. Writing each subset $\Omega_x$ as $R_x \,\cap\, \Omega$ where
  $R_x$ is a sufficiently small neighborbood of $x$ in $\R^d$, we claim that
  $\RLCT_\Omega(f;\varphi) = \min_{x \in \Omega} \RLCT_{R_x}(f;\varphi)$ if this
  minimum is attained in the interior of $\Omega$. Indeed, for $x$ in the
  interior of $\Omega$, we get $\RLCT_{R_x} (f;\varphi) =
  \RLCT_{\Omega_x}(f;\varphi)$. Otherwise, the
  volume of $\{\omega \in \Omega_x : f(\omega) \leq \lambda\}$ is
  less than that of $\{\omega \in R_x : f(\omega) \leq \lambda\}$
  for all $\lambda$. Hence $\RLCT_{R_x} (f; \varphi) \leq \RLCT_{\Omega_x}
  (f;\varphi)$, and the claim follows.

Now to prove Theorem \ref{thm:RLCTMonomial2},
 it suffices to show that for each $x \in
\Omega$ we have
$$\RLCT_{R_x}(f;\varphi) \,\,=\,\, \left( \,\min_{i \in \mathcal{I}_x} \frac{\tau_i+1}{\kappa_i}, \left|
\biggl\{
\operatorname*{argmin}_{i \in \mathcal{I}_x} \frac{\tau_i+1}{\kappa_i} \biggr\}
\right| \,  \right)$$
where $\mathcal{I}_x$ is the set of all indices $0\leq i \leq r$ that
satisfy $\omega_i(x) = 0$. Without loss of generality, suppose $x = (x_1, \ldots,
x_d)$ where $x_1= \cdots = x_{s} = 0$ and $x_{s+1}, \ldots, x_r$ are
nonzero. If $g(x) \neq 0$, we may divide $f(\omega)$ by $g(\omega)$
without changing the RLCT in a sufficiently small neighborhood $R_x$
of $x$. The RLCT of the remaining monomial is determined by
\cite[Proposition 3.7]{LinThesis}. Now, let us suppose $g(x)=0$.
Because $g(\omega)$ is normal crossing with $\omega_1, \ldots,
\omega_r$, one of the derivatives $\partial g/\partial \omega_j$ must
be nonzero at $x$ for
some $s+1\leq
j \leq d$. We assume $R_x$ is sufficiently small so that this derivative
and $\omega_{s+1}, \ldots, \omega_r$ do
not vanish. Consider the map $\sigma:R_x \rightarrow \R^d$ given by
$$
\sigma_j(\omega) = g(\omega), \quad 
\sigma_i(\omega) = \omega_i \quad \text{for }i\neq j.
$$
The Jacobian matrix of $\sigma$ is nonsingular, so this map is an
isomorphism onto its image. Set $U = \mu(R_x)$ and $\rho =
\sigma^{-1}: U \rightarrow R_x$. Then, for all $\mu \in U$, we have
$$
(f \circ \rho)(\mu) \,=\, \mu_1^{\kappa_1} \cdots \mu_{s}^{\kappa_{s}}
\mu_j \cdot a(\mu) \quad\,\, \hbox{and} \quad\,\, (\varphi \circ \rho)(\mu) 
\,=\, \mu_1^{\tau_1} \cdots \mu_s^{\tau_s} \cdot b(\mu),
$$
where the factors $a(\mu)$ and $b(\mu)$ do not vanish in $U$. By using the chain rule \cite[Proposition
4.6]{LinThesis}, we get $\RLCT_{R_x}(f;\varphi) = \RLCT_U(f \circ
\rho;\varphi \circ \rho)$. The latter RLCT can be computed once
again by dividing out the nonvanishing factors and applying 
\cite[Proposition 3.7]{LinThesis}.
\end{proof}

The hypersurface  $\{f(\omega)=0\}$ may not
satisfy the hypothesis in Theorem \ref{thm:RLCTMonomial2}. 
In that case, we can try to simplify its
singularities via a change of variables $\rho: U
\rightarrow \Omega$. 
With some luck, the transformed hypersurface $\{(f \circ \rho) (\mu)=0\}$ will be described
locally by monomials and the RLCT can be computed using
Theorem \ref{thm:RLCTMonomial2}. More precisely, let $U$ be a
$d$-dimensional real analytic manifold and 
$\rho:U \rightarrow \Omega$ a real analytic map that is \emph{proper},
i.e. the preimage of any compact set is compact.
Then $\rho$ \emph{desingularizes} $f(\omega)$ if it satisfies the
following conditions:
\begin{enumerate}
\item The map $\rho$ is an isomorphism outside the variety $\{\omega \in
  \Omega : f(\omega)=0\}$.
\item Given any $y \in U$, there exists a local chart with coordinates
  $\mu_1, \ldots, \mu_d$ such that 
$$
(f \circ \rho)(\mu) = \mu_1^{\kappa_1} \cdots \mu_d^{\kappa_d} \cdot a(\mu) ,
\quad \det \partial \rho(\mu) = \mu_1^{\tau_1} \cdots
\mu_d^{\tau_d} \cdot b(\mu) 
$$
where $\det \partial \rho$ is the Jacobian determinant, the exponents $\kappa_i,
\tau_i$ are nonnegative integers and the real
analytic functions $a(\mu), b(\mu)$ do not vanish at $y$.
\end{enumerate}
If such a desingularization exists, then we may apply $\rho$ to the volume function
\eqref{eq:firstorder} to calculate the RLCT. Care must be
taken to multiply the
measure $\varphi$ with the Jacobian determinant
$\lvert \det \partial \rho\rvert$ in accordance with the
change-of-variables formula for integrals.

Hironaka's celebrated theorem on the resolution of singularities \cite{Hauser, Hironaka}
guarantees that such a desingularization exists for all real
analytic functions $f(\omega)$. The proof employs transformations
known as \emph{blowups} to simplify the singularities. We now describe
the blowup $\rho: U
\rightarrow \R^d$ of the origin in $\R^d$. The manifold $U$ can be
covered by local charts $U_1, \ldots, U_d$ such that each chart is isomorphic to $\R^d$ and each restriction $\rho_i:U_i \rightarrow \R^d$ is the monomial map 
$$(\mu_1, \ldots, \mu_{i-1}, \xi, \mu_{i+1},
\ldots, \mu_d) \,\mapsto\, (\xi\mu_1, \ldots, \xi\mu_{i-1}, \xi, \xi\mu_{i+1}, \ldots, \xi\mu_d).$$
Here, the coordinate hypersurface $\xi=0$, also called the
\emph{exceptional divisor}, runs through all the charts. If the origin
is locally the intersection of many smooth
hypersurfaces with distinct tangent hyperplanes, then these
hypersurfaces can be separated by blowing up the origin
\cite{Hauser}. 

\begin{ex} \label{ex:fourlines}
Consider the curve $\{f(x,y) =xy(x+y)(x-y) = 0\}$ in
Figure \ref{fig:tubes_4}. 
 To resolve this singularity,
we blow up the origin. In the first chart, the map is
$\rho_1 : (\xi, y_1) \mapsto (\xi, \xi y_1)$, so
$$
f \circ \rho_1 \,\,= \,\, \xi^4 y_1(1+y_1)(1-y_1) \quad \hbox{and} \quad \det \partial \rho_1 = \xi.
$$
The lines $\{y=0\}$, $\{x+y = 0\}$ and $\{x-y = 0\}$ are transformed to
$\{y_1=0\}$, $\{y_1=-1\}$ and $\{y_1 = 1\}$ respectively in this chart, thereby separating
them. The line $\{x=0\}$ does not show up here, but it
 appears  as $\{x_1=0\}$ in the second chart, where $\rho_2 : (x_1,
\xi) \mapsto (\xi x_1, \xi)$ and
$$
f \circ \rho_2 =  \xi^4 x_1(x_1+1)(x_1-1), \quad \det \partial \rho_2 = \xi.
$$
Since the curve $\{(1+y_1)(1-y_1)=0\}$ is normal crossing with $\,\xi
y_1\,$ in the first chart, we can now apply  Theorem
\ref{thm:RLCTMonomial2}. The chain rule \cite[Proposition
4.6]{LinThesis} shows that $\RLCT_\Omega(f;1)$ is the minimum of $\RLCT_{U_i}(f\circ \rho_i;
\det \partial \rho_i)$ for $i=1,2$. In both charts, this RLCT equals
$(\frac{1}{2},1)$. \qed
\end{ex}

\begin{ex}
\label{ex:4charts}
Let $p = 4$ and $G$ be the almost complete DAG with $a_{13} = 0$.
We consider the
 conditional independence statement $1 \independent 3\given 4$.
The correlation hypersurface is defined by
$$  f \,\,= \,\,{\rm det}(K_{12,23}) \,\,= \,\,
a_{14} a_{23}^2 a_{34}+a_{14} a_{23} a_{24}+ a_{12} a_{24} a_{34} -a_{12} a_{23}+a_{14} a_{34}.
$$
The real singular locus is a line in the parameter space $\R^5$, since
$\, {\rm Singu}_{1,3,2} \, = \, \langle  a_{12}, a_{14}, a_{23}, a_{34} \rangle $.
Blowing up   this line in $\R^5$ creates four charts $U_1,U_2,U_3,U_4$.
 For instance, the first chart has
$$ \rho_1 \,: \,U_1 \rightarrow \R^5 \, , \,\,(\xi, \mu_{14}, \mu_{23}, a_{24},\mu_{34})
\mapsto (\xi, \xi \mu_{14}, \xi \mu_{23}, a_{24},\xi \mu_{34}), \quad \det \partial \rho_1 = \xi^3.$$
Then $f$  transforms to
$ \,f \circ \rho_1 = \xi^2 \cdot g\,$ where
$\,g = \mu_{14} \mu_{23}^2 \mu_{34}  \xi^2+\mu_{14} \mu_{34} +\mu_{14} \mu_{23} a_{24}+\mu_{34}  a_{24} -\mu_{23} $.
The hypersurface
$\{ g = 0\}$ has no real singularities, so it is smooth in $U_1$.
We can thus apply  Theorem \ref{thm:RLCTMonomial2}
with $\mathcal{I} = \{0,1\}$ to find $\,{\rm RLCT}_{U_1}( \xi^2 \cdot g , \xi^3 ) = (1,1) $.
The behavior is the same on $U_2, U_3$ and $U_4$, and we conclude that
$\,{\rm RLCT}(1,3|4) \,=\,(\ell,m) \,=\, (1,1)$. 
This example is one of the $12$ cases that were labeled as ``Blowup'' in
 Table \ref{table2}. The other $11$ cases are similar.
\qed
\end{ex}

\begin{ex}
\label{ex:tripartite_blowup}
In Example \ref{ex:tripart_1third} we claimed
that ${\rm RCLT}(1,2\,|\,5) = (1,3)$ for  $G = {\rm Tripart}_{5,1}$.
We now prove this claim by using the blowing up method.
 The polynomial in question is
$$ f \,\, = \,\, {\rm det}(K_{1,2|5}) \,\, = \,\, 
(a_{13} a_{35} + a_{14} a_{45}) (a_{23} a_{35} + a_{24} a_{45}).
$$
The singular locus of the hypersurface $\{f = 0\}$ is given by
$$ {\rm Singu}_{1,2,34} \,\, = \,\,
\langle \,a_{35}\,, \,a_{45}\, \rangle \,\,\cap \,\,
\bigl\langle \, \hbox{$2 \times 2$-minors of} \,
\begin{pmatrix}
a_{13} &  a_{23} & a_{45}  \\
 a_{14} & a_{24} & \!\!\! -a_{35} \\
\end{pmatrix} \,
\bigr\rangle.
$$
We blow up the linear subspace $\{a_{35} = a_{45} = 0\}$ in $\R^6$.
This creates two charts. The map for the first chart is $\,\rho_1:
(a_{13},a_{14},a_{23},a_{24}, \xi ,\mu_{45}) \mapsto
(a_{13},a_{14},a_{23},a_{24},\xi,\xi \mu_{45})$. This map gives
$$
f\circ \rho_1 = \xi^2 (a_{13}  + a_{14} \mu_{45}) (a_{23} + a_{24}
\mu_{45}), \quad \det \partial \rho_1 = \xi.
$$
Now, by setting $a_{13} = x-a_{14}\mu_{45}$ and $a_{23} = y -
a_{24}\mu_{45}$, the transformed function $f \circ \rho_1$ is the monomial 
$\xi^2 x y$. Then Theorem \ref{thm:RLCTMonomial2} can be employed 
to evaluate $\,\RLCT_{U_1}(\xi^2 x y, \xi) = (1,3)$. The calculation in the second
 chart is completely analogous.
\qed
\end{ex}

The same approach as in Example \ref{ex:tripartite_blowup}
can be applied to the polynomial
$\,f = {\rm det}(K_{1,2|56})\,$ in Example \ref{ex:tripartite6}.
A lengthy calculation, involving many charts and multiple blowups,
eventually reveals that $G = {\rm Tripart}_{6,2}$  satisfies
$\,{\rm RCLT}(1,2|56) = (1,1)$.
This was stated  in Example \ref{ex:tripart_62}.

\section{Computing the constants}
\label{sec:integrals}

We now describe a method for finding the constant $C$ in
the formula $V(\lambda) \approx
C\lambda^{-\ell}(-\ln \lambda)^{m-1}$ in (\ref{eq:firstorder}). The
two theorems in this section are new and they extend the results of Greenblatt \cite{Greenblatt}
and Lasserre \cite{Lasserre} on the volumes of sublevel sets. Unless stated otherwise, all measures used in this section are
the standard Lebesgue measures.
We begin by showing that the constant $C$ is
 a function of the highest order term in the Laurent expansion of the zeta function of $f$.

\begin{lem} \label{thm:volumeConstant} Given real analytic functions $f, \varphi:\Omega
  \rightarrow \R$, consider the Laurent expansion of
$$
\zeta(z) := \int_{\Omega} |f(\omega)|^{-z} \varphi(\omega)d\omega \,\,=\,\,
\frac{a_{\ell,m}}{(\ell-z)^m} + 
\frac{a_{\ell,m-1}}{(\ell-z)^{m-1}} + 
\cdots
$$
where $\ell$ is the smallest pole and $m$ its
multiplicity. Then, asymptotically as $\lambda$ tends to zero, 
$$
V(\lambda) := \int_{|f(\omega)|\leq \lambda} \varphi(\omega)d\omega
\,\,\approx\,\, \frac{a_{\ell,m}}{\ell(m-1)!} \,\lambda^{\ell} (-\ln \lambda)^{m-1}.
$$
\end{lem}
\begin{proof}
According to the proof of \cite[Theorem 7.1]{Watanabe09}, the volume function
$V(\lambda)$ equals $\int_0^\lambda v(s)ds$
where $\,v(s)=\int_\Omega \delta(s-\!f(\omega))\,\varphi(\omega)d\omega\,$ is the state density function
and $\delta$ is the delta function. Now, using the proof of
\cite[Theorem 3.16]{LinThesis}, we obtain
$$
v(s) =  \frac{a_{\ell,m}}{(m-1)!}\, s^{\ell-1} (-\ln s)^{m-1} +\,
o(s^{\ell-1} (-\ln s)^{m-1}) \quad \text{as }s \rightarrow 0.
$$
Here we used the little-o notation. Finally, using integration by
parts, we find that
$$
V(\lambda) =  \frac{a_{\ell,m}}{\ell (m-1)!}\, \lambda^{\ell} (-\ln
\lambda)^{m-1} +\, o(\lambda^{\ell} (-\ln
\lambda)^{m-1}) \quad \text{as }\lambda \rightarrow 0. \qedhere
$$
\end{proof}

\begin{ex}
In Example \ref{ex:d-ball}, we saw that the volume of the
$d$-dimensional ball defined by $\,|\omega_1^2+\cdots+\omega_d^2| \leq
\lambda\,$ is equal to $\,V(\lambda) = C \lambda^{-d/2}\,$ for some positive constant
$C$. We here show how to compute that constant using asymptotic
methods. By Lemma
\ref{thm:volumeConstant}, $C = 2\alpha/(d\,2^d)$~where $\alpha$ is the coefficient of
$(d/2-z)^m$ in the Laurent expansion of the zeta function
$$
\zeta(z) \,=\, \int_{\R^d} |\omega_1^2+\ldots+\omega_d^2|^{-z} d\omega.
$$
Computing this Laurent coefficient from first principles is not
easy. Instead, we derive $\alpha$~using the asymptotic theory of
Laplace integrals. The connection between such integrals and volume
functions was alluded to in Definition \ref{thm:DefRLCT}. By
\cite[Proposition 5.2]{LinThesis}, the Laplace integral
$$
Z(N)\, =\, \int_{\R^d} e^{-N(\omega_1^2+\cdots+\omega_d^2)} d\omega
$$
is asymptotically $\alpha \,\Gamma(\frac{d}{2}) N^{-d/2}$ for large $N$. But this
Laplace integral also decomposes as
$$
Z(N) \, =\, \int_\R e^{-N\omega_1^2}d\omega_1 \cdots \int_\R
e^{-N\omega_d^2}d\omega_d \,\,=\,\, (\sqrt{\pi} N^{-1/2})^d,
$$
where each factor is the classical Gaussian integral. Solving for
$\alpha$ leads to the formula 
$$
C \,=\, \frac{\pi^{d/2} }{2^d \cdot \Gamma(\frac{d}{2}) \cdot \frac{d}{2}}
\,=\, \frac{\pi^{d/2} }{2^d \cdot \Gamma(\frac{d}{2}+1)}.
$$
\vskip -0.5cm
\qed
\end{ex}

In Section \ref{sec:real_log_threshold}, we saw how the RLCTs of smooth
hypersurfaces and of hypersurfaces defined by monomial
functions can be computed. The following two theorems  and  their
accompanying examples
demonstrate how the asymptotic constant $C$ can also be evaluated in
those instances. Here, we say that two hypersurfaces \emph{intersect
  transversally} in $\R^d$  if the points of intersection are smooth
on the hypersurfaces and if the corresponding tangent spaces at each intersection point generate
the tangent space of $\R^d$ at that point.

\begin{thm}\label{thm:smoothConstant} Let $\{f = 0\}$ be a smooth
  hypersurface and let $\varphi:\Omega \rightarrow \R$
    be positive. Suppose $\partial f/ \partial {\omega_1}$ is
  nonvanishing in
    $\Omega$. Let $W$
    be the projection of the hypersurface $\{f = 0\} \subset \Omega$ onto the subspace $\{(\omega_2, \ldots,
    \omega_d) \in \R^{d-1}\}$ and let $\rho:\Omega \rightarrow \R^d$ be the map $\omega
    \mapsto (f(\omega), \omega_2, \ldots, \omega_d)$. If the boundary of $\Omega$ intersects transversally with
  the hypersurface $\{f=0\}$, then 
$$ 
V(\lambda) := \int_{\{\omega \in \Omega : |f(\omega)| \leq \lambda\}}
\varphi(\omega) d\omega \,\,\approx\,\, C \lambda
$$
asymptotically (as $\lambda\to 0$), where 
$$C \,\,=\,\, 2\int_{W} \frac{\varphi}{|\partial_{\omega_1}f|} \circ
\rho^{-1} (0, \omega_2, \ldots, \omega_d)\, d\omega_2 \cdots
d\omega_d.$$
\end{thm}
\begin{proof} The asymptotics of the volume $V(\lambda)$ depends only on the
  region $\{\omega \in \Omega: |f(\omega)|\leq \lambda\}$. So we may
  assume that $\Omega$ is a small neighborhood of the hypersurface
  $\{f(\omega)=0\}$. As we saw in the proof of 
  Theorem \ref{thm:RLCTMonomial2},
    the map $\rho$ is an isomorphism onto its image. Thus
  after changing variables, the zeta function associated to $V(\lambda)$ becomes
\begin{align*}
\zeta(z) &=  \int_{\rho(\Omega)} |f|^{-z}
\frac{\varphi}{|\partial_{\omega_1} f|} \circ \rho^{-1}(f, \omega_2,
\ldots, \omega_d) \,\,df d\omega_2\cdots d\omega_d \\
&=\int_{W} \int_{\varepsilon_1(\omega_2, \ldots, \omega_d)}^{\varepsilon_2(\omega_2, \ldots, \omega_d)} |f|^{-z}
\frac{\varphi}{|\partial_{\omega_1} f|} \circ \rho^{-1}(f, \omega_2,
\ldots, \omega_d) \,\,df d\omega_2\cdots d\omega_d.
\end{align*}
Here, the lower and upper limits $\varepsilon_1, \varepsilon_2$ straddle
 zero because the boundary of $\Omega$ is transversal to the hypersurface.
By substituting the Taylor series
$$
\frac{\varphi}{|\partial_{\omega_1} f|} \circ \rho^{-1}(f, \omega_2,
\ldots, \omega_d) = \frac{\varphi}{|\partial_{\omega_1} f|} \circ \rho^{-1}(0, \omega_2,
\ldots, \omega_d) + O(f)
$$
and the exponential series $\varepsilon_2^{1-z} = 1 + O(1-z)$, we get the Laurent expansion
\begin{align*}
\int_{0}^{\varepsilon_2} |f|^{-z}
\frac{\varphi}{|\partial_{\omega_1} f|} \circ \rho^{-1}(f, \omega_2,
\ldots, \omega_d) \,\,df &= \left[ \frac{|f|^{1-z}}{1-z}\cdot \frac{\varphi}{|\partial_{\omega_1} f|} \circ \rho^{-1}(0, \omega_2,
\ldots, \omega_d) \right]_0^{\varepsilon_2} + \cdots \\
&=\frac{1}{1-z}\cdot  \frac{\varphi}{|\partial_{\omega_1} f|} \circ \rho^{-1}(0, \omega_2,
\ldots, \omega_d) + \cdots.
\end{align*}
The same is true for the integral from $\varepsilon_1$ to
$0$. The result now follows from Lemma~\ref{thm:volumeConstant}.
\end{proof}

\begin{ex} By Theorem \ref{thm:uptosix}, all conditional
  independence statements in small complete graphs lead to smooth
  hypersurfaces. Here we analyze the statement $\,1\independent 2\given 3\,$ in the complete 3-node DAG. This example was studied in \cite[\S 2]{URBY}. The corresponding partial correlation is
$$\corr(1,2\,|\, 3)\,=\,\frac{a_{13}a_{23}-a_{12}}{\sqrt{1+a_{23}^2}\sqrt{1+a_{12}^2+a_{13}^2}}. $$
This partial correlation hypersurface lives
in $\R^3$ and it is depicted in \cite[Figure 2(b)]{URBY}. 

We apply Theorem \ref{thm:smoothConstant} by setting
$\Omega:=[-1,1]^3$, $f:=\corr(1,2\,|\, 3)$ and
$\varphi:=1/2^3$, the uniform distribution on $\Omega$. We
choose $\omega_1$ to be $a_{12}$. Then $\rho^{-1}(0,a_{13},a_{23}) =
(a_{13}a_{23},a_{13},a_{23})$. The projection $W$ of the
surface $\,\{a_{12} = a_{13}a_{23}\}\,$ onto $\,\{ (a_{13},a_{23}) \in
[-1,1]^2 \}\,$ is the whole square $[-1,1]^2$. The formula for the constant $C$ now
simplifies to
$$
C \,\,=\,\,
\frac{1}{4}\int_{-1}^1\int_{-1}^1\sqrt{1+a_{13}^2}\sqrt{1+a_{13}^2+a_{13}^2a_{23}^2}\,\,da_{13}\,da_{23}
\,\,\approx\,\, 5.4829790759.
$$
This two-dimensional integral was evaluated numerically using
\texttt{Mathematica}. \qed
\end{ex}

We now come to the monomial case that was discussed
in Theorem \ref{thm:RLCTMonomial2}.

\begin{thm} 
\label{thm_constants}
Let $g:\Omega \rightarrow \R$ and $\varphi:\Omega \rightarrow \R$ be
positive, and let $f(\omega) = \omega_1^{\kappa_1} \cdots
\omega_d^{\kappa_d} \,g(\omega)$ where the $\kappa_i$ are nonnegative integers. Suppose that 
$1/\ell = \kappa_1 = \cdots = \kappa_m > \kappa_{m+1} \geq \cdots \geq
\kappa_d$ and that the boundary of $\Omega$ is transversal to the
subspace $L$ defined by $\omega_1 = \cdots = \omega_m =
0$. Let $\bar{\omega}$ and $\bar{\kappa}$ denote the vectors $(\omega_{m+1},\ldots, \omega_d)$ and $(\kappa_{m+1},\ldots,
\kappa_d)$ respectively. Then 
$$ 
V(\lambda) := \int_{\{\omega \in \Omega : |f(\omega)| \leq \lambda\}}
\varphi(\omega) d\omega \,\,\approx\,\, C \lambda^{\ell} (-\ln \lambda)^{m-1}
$$
asymptotically as $\lambda$ tends to zero where 
\begin{equation}
\label{eq:constant}
C\, =\, \frac{(2\ell)^m}{\ell(m-1)!}
\int_{\Omega \,\cap\, L} \bar{\omega}^{-\ell\bar{\kappa}} g(0,\ldots,0,\bar{\omega})^{-\ell}
\varphi(0,\ldots, 0,\bar{\omega}) d\bar{\omega}.
\end{equation}
\end{thm}

\begin{proof} Let us suppose for now that $\Omega$ is the
  hypercube~$[0,\varepsilon]^d$. Our goal is to apply
  Lemma~\ref{thm:volumeConstant} by computing the Laurent
  coefficient $a_{\ell, m}$ of the zeta function $\zeta(z)$. We first
  study the Taylor series expansion of the integrand about
  $\omega_1=\cdots = \omega_m = 0$. This gives
$$
\left(\omega^{\kappa} 
  g(\omega)\right)^{-z}\varphi(\omega) = \omega^{-z\kappa}
\left(g(0,\ldots,0,\bar{\omega})^{-z} \varphi(0, \ldots, 0, \bar{\omega})+ O(\omega_1) + \cdots + O(\omega_m)\right).
$$
The higher order terms in this expansion contribute larger
poles to $\zeta(z)$, so we only need to compute the coefficient
of $(\ell-z)^{-m}$ in the Laurent expansion of
\begin{align*}
&\int_\Omega \omega_1^{-z\kappa_1}\cdots \omega_m^{-z\kappa_m} \, \bar{\omega}^{-z\bar{\kappa}}\,
g(0,\ldots,0,\bar{\omega})^{-z} \,\varphi(0, \ldots, 0, \bar{\omega})
\,\,d\omega_1 \cdots d\omega_m \,d\bar{\omega}\\
&= \left(\frac{\varepsilon^{1-z/\ell}}{1-z/\ell}\right)^m  \int_{\Omega \,\cap\, L} \bar{\omega}^{-z\bar{\kappa}}\,
g(0,\ldots,0,\bar{\omega})^{-z} \,\varphi(0, \ldots, 0, \bar{\omega})
\,\,d\bar{\omega}.
\end{align*}
Because $g$ is positive, the last integral in the above expression
does not have any poles near $z = \ell$, so the constant term in its
Laurent expansion comes from substituting $z = \ell$. Hence, 
$$
a_{\ell,m} = \ell^m \int_{\Omega \,\cap\, L} \bar{\omega}^{-\ell\bar{\kappa}}\,
g(0,\ldots,0,\bar{\omega})^{-\ell} \,\varphi(0, \ldots, 0, \bar{\omega})
\,\,d\bar{\omega}.
$$
Now suppose $\Omega$ is not the hypercube. Since the boundary of $\Omega$ is transversal to the
subspace~$L$, we decompose $\Omega$ into small neighborhoods which are
isomorphic to orthants. Summing up~the contributions from these orthants gives the desired result.
\end{proof}

\begin{rem}
We revisit the planar tubes shown in Figures
\ref{fig:tubes_1}--\ref{fig:tubes_3}. Using 
the formula~(\ref{eq:constant}) in
Theorem  \ref{thm_constants}, one can easily check the constants $C$ we saw in Example \ref{ex_tubes},
namely
$$ 
C= \begin{cases}
\,1&\textrm{for } f(x,y)=x,\\
\,1& \textrm{for } f(x,y)=xy,\\
\,3& \textrm{for } f(x,y)=x^2y^3.
\end{cases}
$$
\end{rem}

\begin{ex} \label{ex:chainStarConstants}
We apply Theorem \ref{thm_constants} to find the constants in
Corollary \ref{cor__const_trees} for chains
and stars. In both cases we set  $\Omega = [-1,1]^{p-1}$ and
$\varphi=2^{1-p}$. For chains we have $(\ell,m)=(1,p-1)$ and $L$ is the subspace $a_{12} = \cdots =
a_{p-1,p} = 0$. Then the integral in (\ref{eq:constant}) is the
evaluation of the denominator of (\ref{eq:chaincor}) at the origin
multiplied by $\varphi$, so $C_{\rm chain} \, = \, 1/(p-2)!$ as
claimed. 

For stars, $(\ell,m)=(1,2)$ is achieved by $1<i<j$ and $S \subset
\bar{S} :=
\{2, \ldots, p\} \setminus \{i,j\}$ with 
$$
\corr(i,j|S) = -\frac{a_{1i}a_{1j}}{\sqrt{1+SOS(S)
    +a_{1i}^2}\,\sqrt{1+SOS(S)+a_{1j}^2}}, \quad SOS(S) = \sum_{s \in S} a_{1s}^2.
$$
Since
$|\corr(i,j|S)| \geq |\corr(i,j|\bar{S})|$, the quantity
$V_G(\lambda)$ is the volume of the union of the tubes $\{|\corr(i,j|\bar{S})|
\leq \lambda\}$ over all $1{<}i{<}j$. By applying formula
(\ref{eq:constant}), the asymptotic
volume of each
tube computes to $p\lambda (-\ln \lambda)/3$. Meanwhile, the volumes of the
intersections of these tubes become negligible as $\lambda
\rightarrow 0$. After summing over all $1{<}i{<}j$,
 we get $C_{star} = \binom{p-1}{2} \frac{p}{3} = \binom{p}{3}$.
\qed
\end{ex}

\begin{ex} \label{ex:bowtieConstant} 
We compute the constant $C$ of the volume
  $V_{1,2|4,\ldots,p}(\lambda)$ for $\,G = {\rm
    Tripart}_{p,p-3} \,$ as in Example \ref{ex:tripart_1}. Let $p \geq
  6$ and   $\Omega$ be given by
  (\ref{eq:newParamSpace}). We are interested in the tube
$$
\left|a_{13}a_{23}\,\frac{g(\bar{a})}{h(a)}\right| \leq \lambda, \quad \text{where }g(\bar{a}) =\sum_{k=4}^p
  a_{3k}^2  \,\,
  \hbox{ and } \,\,
  h(a) = \sqrt{1+g(\bar{a}) (a_{13}^2{+}1)} \cdot \sqrt{1+g(\bar{a})
    (a_{23}^2{+}1)}.
$$
The measure on $\Omega$ is $\varphi(a)\,da_{12}\,da_{23}\,d\bar{a}$
where $\varphi(a) = 1/4$ and $d\bar{a}$ is the Lebesgue probability measure on the ball $\{g(\bar{a})
\leq 1\}$. According to Theorem
\ref{thm_constants}, 
$$ 
C = 
  \int_{\{g(\bar{a})\leq 1\}} \!\!\left(\frac{g(\bar{a})}{1+g(\bar{a})}
  \right)^{-1}\!\! d\bar{a}
$$
By substituting spherical coordinates for the integration,
this expression simplifies to 
$$
1+ \int_{\{g(\bar{a})\leq 1\}} \frac{1}{g(\bar{a})}\, d\bar{a}\,\,
=\,\, 2 + \frac{2}{p-5}\, ,
$$
yielding the constant $C_{1,2|4,\ldots,p} $ needed for the bias reduction analysis in
Example \ref{ex:tripart_1}.
\qed
\end{ex}

\section{Discussion}

In this paper we examined the volume of regions in the parameter space
of a directed Gaussian graphical model that are given by bounding partial
correlations. We established a connection to singular learning theory,
and we showed that these volumes can be computed by evaluating the real
log canonical threshold of the partial correlation
hypersurfaces. Throughout the paper we have made the simplifying
assumption of equal noise,
i.e.~$\epsilon\sim\mathcal{N}(0,I)$. Ideally, one would like to allow
for different noise variances. This would increase the dimension of
the parameter space $\Omega$. It would be very interesting to study
this more difficult situation and understand how the asymptotic
volumes change, or more generally, how
the asymptotics depends on our choice of the parameter space $\Omega$.
This issue
was discussed briefly in Example \ref{ex:specialcases}.

This paper can be seen as a first step towards developing a theory which would allow to compute the complete asymptotic expansion of particular volumes. We have concentrated on computing only the leading coefficients of these expansions, and even this question is still open in many cases (e.g.~Example 6.6). An interesting extension would be to better understand how to use properties of the graph to compute the coefficients $C_{l,m}$ in the asymptotic expansion. Finally, another interesting problem for future research would be to ascertain all values of $(l,m)$ for which $C_{l,m}$ is non-zero, in terms of intrinsic properties of the underlying graph.

\section*{Acknowledgments}
This project began at the workshop ``Algebraic Statistics in the
Alleghenies'', which took place at Penn State
in June 2012. We are grateful to the organizers for a very inspiring meeting. We thank Thomas Richardson for pointing out the
connection to the problem of bias reduction in causal inference. We
also thank all reviewers for thoughtful comments and one of the reviewers in particular for numerous detailed comments and
 for spotting a mistake in a proof.~This
work was supported in part by the US National Science Foundation  (DMS-0968882)
and the DARPA Deep Learning program (FA8650-10-C-7020).


\begin{thebibliography}{10}

\bibitem{AGV} V.~I.~Arnol'd, S.~M.~Guse\u\i n-Zade and A.~N.~Varchenko: {\em Singularities of Differentiable Maps}, Vol. II, Birkh\"auser, Boston, 1985.

\bibitem{Chaudhuri}
S.~Chaudhuri and T.S.~Richardson:
Using the structure of d-connecting paths as a qualitative measure of the strength of dependence, {\em Proceedings of the Nineteenth Conference on Uncertainty in Artificial Intelligence}, pages 116--123, 2003.

\bibitem{CLO}
D.~Cox, J.~Little and D.~O'Shea:
{\em  Ideals, Varieties and Algorithms},
Springer Undergraduate Texts,  Third edition,
Springer Verlag, New York, 2007.

\bibitem{Singular}
W.~Decker, G.-M.~Greuel, G.~Pfister, and H.~Sch{\"o}nemann:
\newblock {\sc Singular} {3-1-5} -- {A} computer algebra system for polynomial computations,
2012, \url{http://www.singular.uni-kl.de}.

\bibitem{M2} D.~Grayson and M.~Stillman:
{\sc Macaulay 2}, a software system for research in algebraic geometry,
available at \url{http://www.math.uiuc.edu/Macaulay2/}.

\bibitem{Greenblatt} M.~Greenblatt: Oscillatory integral decay, sublevel set growth,
  and the Newton polyhedron, {\em Math. Annalen} {\bf 346}, pages 857--890, 2010.

\bibitem{Greenland}
S.~Greenland:
Quantifying biases in causal models: classical confounding vs collider-stratification bias, {\em Epidemiology}  {\bf 14}, pages 300--306, 2003.

\bibitem{GP}
S.~Greenland and J.~Pearl:
Adjustments and their consequences: collapsibility analysis using graphical models, {\em International Statistical Review}  {\bf 79}, pages 401--426, 2011.

\bibitem{Hauser} H.~Hauser: The Hironaka theorem on resolution of
  singularities (or: a proof that we always wanted to
  understand), {\em Bull. Amer. Math. Soc.} {\bf 40}, pages 323--403, 2003.

\bibitem{Hironaka} H.~Hironaka: Resolution of singularities of an
  algebraic variety over a field of characteristic zero. I,
  II. {\em Ann. of Math.} (2) {\bf 79}, pages 109--326, 1964.

\bibitem{kabu07} M. Kalisch and P. B\"uhlmann: {Estimating high-dimensional directed acyclic graphs with the PC-algorithm, {\em Journal of Machine Learning Research} {\bf 8}, pages 613--636, 2007.

\bibitem{kollar} J.~Koll\'{a}r: {\em Lectures on Resolution of
 Singularities}, Annals of Mathematics Studies {\bf 166}, Princeton University Press, Princeton, NJ, 2007. 

\bibitem{Lasserre} J.~B.~Lasserre: Level sets and non Gaussian
    integrals of positively homogeneous functions, {\tt  arXiv:1110.6632}.

\bibitem{LinThesis} S.~Lin: {\em Algebraic Methods for Evaluating Integrals in Bayesian Statistics}, Ph.D. dissertation, University of California, Berkeley, May 2011.

\bibitem{Marshall}
M.~Marshall: {\em Positive Polynomials and Sums of Squares},
Mathematical Surveys and Monographs 
{\bf 146}, American Mathematical Society, Providence, RI, 2008.

\bibitem{SGS} P.~Spirtes, C.~Glymour and R.~Scheines: {\em Prediction and
      Search}, MIT Press, second edition, 2001.

\bibitem{URBY}
 C.~Uhler, G.~Raskutti, P.~B\"uhlmann and B.~Yu:
Geometry of faithfulness assumption in causal inference, {\em Annals of Statistics} {\bf 41}, pages 436--463, 2013.

\bibitem{Watanabe09} S.~Watanabe: {\em Algebraic Geometry and
    Statistical Learning Theory},  Monographs on Applied and
  Computational Mathematics {\bf 25}, Cambridge University Press,  2009.

\bibitem{ZhangSpirtes03} J.~Zhang and P.~Spirtes: Strong faithfulness and uniform consistency in causal inference, {\em Uncertainty in Artificial Intelligence (UAI)}, pages 632--639, 2003.

\bibitem{ZhangSpirtes08} J.~Zhang and P.~Spirtes: 
Detection of unfaithfulness and robust causal inference, {\em Minds and Machines} {\bf 18}, pages 239--271, 2008.
}

 
\end{thebibliography}
\end{document}